%% file: submit.tex
\documentclass[11pt, twoside]{amsart}

\title{On the Monoidal Center of Deligne's Category \texorpdfstring{$\uRep(S_\MakeLowercase{t})$}{Rep(St)}}
\date{\today}
\author{Johannes Flake}
\address{Lehrstuhl B f\"ur Mathematik (Algebra), RWTH Aachen University,
Pontdriesch 10-16,
52062 Aachen, Germany}
\email{flake@mathb.rwth-aachen.de}

\author{Robert Laugwitz}
\address{School of Mathematical Sciences, University of Nottingham,
University Park,
Nottingham, NG7 2RD,
United Kingdom}
\email{robert.laugwitz@nottingham.ac.uk}
\urladdr{https://www.nottingham.ac.uk/mathematics/people/robert.laugwitz}

\usepackage{import}
\usepackage{amsmath,amsfonts,amsthm,amssymb}
\usepackage[alphabetic, initials]{amsrefs}
\usepackage[english]{babel}
\usepackage{url}
\usepackage{fancyhdr}
\usepackage{graphicx}
\usepackage{microtype}
\usepackage[
colorlinks=true,
linkcolor=black,
anchorcolor=black,
citecolor=black,
urlcolor=black,
]{hyperref}
\usepackage{cleveref}
\usepackage{geometry}
\usepackage[all]{xy}
\usepackage{tikz}
\usetikzlibrary{positioning}

\tikzset{
 mytikzlink/.style={
  line width=0.6pt,
  preaction={draw, {}-{}, line width=3pt, white}
 }
}

\newcommand\torusknotinv[2]{
 \tikz[scale=0.35, yscale=-1,
  baseline={([yshift=-0.35ex]current bounding box.center)}, looseness=1.2] {
 \useasboundingbox (0,0) circle (2.2);
 \def\n{#2} \def\w{360/\n} \def\dw{#1*\w/2}
 \foreach \pm in {-1,1} {
 \foreach \i in {1,...,\n} {
  \draw[mytikzlink] (\i*\w:0.8)
  to[out=\i*\w+\pm*90, in=\i*\w+\pm*\dw-\pm*90]
  (\i*\w+\pm*\dw:2);
 }}
 } 
}

\newcommand\torusknotinvsmall[2]{
 \tikz[scale=0.2, yscale=-1,
  baseline={([yshift=-0.35ex]current bounding box.center)}, looseness=1.2] {
 \useasboundingbox (0,0) circle (2.2);
 \def\n{#2} \def\w{360/\n} \def\dw{#1*\w/2}
 \foreach \pm in {-1,1} {
 \foreach \i in {1,...,\n} {
  \draw[mytikzlink] (\i*\w:0.8)
  to[out=\i*\w+\pm*90, in=\i*\w+\pm*\dw-\pm*90]
  (\i*\w+\pm*\dw:2);
 }}
 } 
}

\tikzset{
	partition/.style={
      scale=0.5,
      yscale=-1,
      baseline={([yshift=-0.5ex]current bounding box.center)}
    }
}

\newcommand\makeDot[2]{\draw[fill] (#1,#2) circle (2.5pt);}
\newcommand\makeDots[2]{\makeDot{#1}{#2} \makeDot{#1}{#2+1}}

\newcommand\tpid[2][1]{%
\foreach \x in {#2} \draw (\x,#1) -- +(0,1);
}

\newcommand\tpart[2]{\tikz[partition]{%
\foreach \i [count=\c] in {#1} \relax
\draw[white] (1,1)--(1,\c); 
\foreach \i [count=\y] in {#1} {
  \ifnum \i>0 \foreach \x in {1,...,\i} \makeDot{\x}{\y};
  \fi
}
#2
}}

\newcommand\tpartx[2]{\tikz[partition]{%
\draw (-1.2,1)--node[right]{...} (-1.2,2) (0,1)--(0,2);
\makeDot{-1.2}{1}\makeDot{-1.2}{2}
\makeDot{0}{1}\makeDot{0}{2}
\foreach \i [count=\c] in {#1} \relax
\draw[white] (1,1)--(1,\c); 
\foreach \i [count=\y] in {#1} {
  \ifnum \i>0 \foreach \x in {1,...,\i} \makeDot{\x}{\y};
  \fi
}
#2
}}

\tikzset{
    bend/.cd,
    0/.style={},
    1/.style={bend right},
    -1/.style={bend left}
}
\newcommand\makePartPt[1]{({Mod(#1,10)},{(#1-Mod(#1,10))*.1})}
\newcommand\makePartLn[2]{%
\pgfmathtruncatemacro\bend{%
(int(#1/10)==int(#2/10) && abs(#1-#2)>1) ?
(#1<10 ? 1 : -1)*(#1>#2 ? 1 : -1)
: 0%
}
\draw ({mod(#1,10)},{int(#1/10}) to[bend/\bend] ({mod(#2,10)},{int(#2/10)});
}
\newcommand\tp[1] {%
\tikz[partition] {
\def\j{0}
\foreach \i [remember=\i as \j] in {#1} {
  \ifnum \i>0
    \ifnum \j>0
      \makePartLn{\i}{\j};
    \fi
    \draw[fill] \makePartPt\i circle (2.5pt);
  \fi
} 
}}

\newcommand{\bigslant}[2]{{\raisebox{.2em}{$#1$}\left/\raisebox{-.2em}{$#2$}\right.}}
\newcommand{\leftexpsub}[3]{{\vphantom{#3}}^{#1}_{#2}{#3}}
\newcommand{\lYD}[1]{\leftexpsub{#1}{#1}{\mathbf{YD}}}
\newcommand{\bunderline}[1]{\mkern1mu\underline{\mkern-1mu#1\mkern-7mu}\mkern7mu }
\newcommand{\uD}[1]{{\mkern1mu\underline{\mkern-1mu D\mkern-4mu}\mkern4mu}_{\mkern-2mu#1}}
\newcommand{\uV}[1]{\bunderline{V}^{#1}}
\newcommand{\uW}[1]{\bunderline{W}_{\mkern-7mu#1}}
\newcommand{\Set}[1]{\left\lbrace #1\right\rbrace}
\newcommand{\isomorph}{\stackrel{\sim}{\longrightarrow}}
\newcommand\inv{^{-1}}

\newcommand{\op}[1]{\operatorname{#1}}

\newcommand{\ov}[1]{\overline{#1}}

\newcommand{\coev}{\operatorname{coev}}
\newcommand{\Drin}{\operatorname{Drin}}
\newcommand{\ev}{\operatorname{ev}}
\newcommand{\End}{\operatorname{End}}
\newcommand{\Hom}{\operatorname{Hom}}
\newcommand{\ide}{\operatorname{Id}}
\newcommand{\Ind}{\operatorname{Ind}}
\newcommand{\reg}{\mathrm{reg}}
\newcommand{\Rep}{\mathrm{Rep}}
\newcommand{\uRep}{\protect\underline{\mathrm{Re}}\mathrm{p}}
\newcommand{\tr}{\operatorname{tr}}
\newcommand{\triv}{\operatorname{triv}}

\providecommand{\cal}[1]{\mathcal{#1}}
\providecommand{\fr}[1]{\mathfrak{#1}}

\newcommand{\mR}{\mathbb{R}}

\newcommand{\mZ}{\mathbb{Z}}
\newcommand{\mN}{\mathbb{N}}

\newcommand{\cC}{\mathcal{C}}
\newcommand{\cD}{\mathcal{D}}
\newcommand{\cF}{\mathcal{F}}
\newcommand{\cI}{\mathcal{I}}

\newcommand{\cL}{\mathcal{L}}
\newcommand{\cT}{\mathcal{T}}
\newcommand{\cN}{\mathcal{N}}
\newcommand{\cQ}{\mathcal{Q}}
\newcommand{\cR}{\mathcal{R}}
\newcommand{\cS}{\mathcal{S}}
\newcommand{\cZ}{\mathcal{Z}}

\newcommand{\rF}{\mathrm{F}}
\newcommand{\rP}{\mathrm{P}}

\numberwithin{equation}{section}

\newtheorem*{rep@theorem}{\rep@title}
\newcommand{\newreptheorem}[2]{%
\newenvironment{rep#1}[1]{%
 \def\rep@title{#2 \ref{##1}}%
 \begin{rep@theorem}}%
 {\end{rep@theorem}}}
\makeatother

\newtheorem{theorem}{Theorem}[section]

\newtheorem{proposition}[theorem]{Proposition}
\newreptheorem{proposition}{Proposition}
\newtheorem{corollary}[theorem]{Corollary}
\newreptheorem{corollary}{Corollary}
\newtheorem{lemma}[theorem]{Lemma}

\newtheorem{theorem*}{Theorem}
\newreptheorem{theorem}{Theorem}

\theoremstyle{definition}
\newtheorem{definition}[theorem]{Definition}

\theoremstyle{remark}
\newtheorem{example}[theorem]{Example}
\newtheorem{remark}[theorem]{Remark}

\newtheorem{question}[theorem]{Question}

\renewcommand{\sectionmark}[1]		
	{
	\markboth{\small\it \thesection{} #1}{}
	}


\begin{document}

\begin{abstract}
We explicitly compute a monoidal subcategory of the monoidal center of Deligne's interpolation category $\uRep(S_t)$, for $t$ not necessarily a natural number, and we show that this subcategory is a ribbon category. For $t=n$, a natural number, there exists a functor onto the braided monoidal category of modules over the Drinfeld double of $S_n$ which is essentially surjective and full. Hence the new ribbon categories interpolate the categories of crossed modules over the symmetric groups. As an application, we obtain invariants of framed ribbon links which are polynomials in the interpolating variable $t$. These polynomials interpolate untwisted Dijkgraaf--Witten invariants of the symmetric groups.
\end{abstract}
\maketitle


\section{Introduction}

Deligne introduced a symmetric monoidal category $\uRep(S_t)$ for any complex number $t$ in \cite{Del} which interpolates the representation categories of symmetric groups $\Rep(S_n)$ for $n\geq1$. Given any monoidal category $\cC$, there exists a general procedure to produce a braided monoidal category --- the \emph{monoidal center} or \emph{Drinfeld center} $\cZ(\cC)$. This center construction has applications to the algebraic construction of ribbon invariants and $3$-dimension topological quantum field theories (TQFTs). In this paper, we investigate the monoidal center of the category $\uRep(S_t)$ and obtain invariants of framed ribbons links from it. We give explicit constructions of ribbon subcategories $\cD_t$ of $\cZ(\uRep(S_t))$, depending on a parameter $t\in \Bbbk$, that interpolate the categories of crossed modules over symmetric groups which correspond to the monoidal centers of the respective representation categories. Throughout this paper, $\Bbbk$ is a field of characteristic zero.

\subsection{Crossed Modules and the Monoidal Center}
The concept of crossed modules over a group $G$ goes back to Whitehead \cite{Whi}*{Section~2} and \cite{EM}*{p.\,56} in the study of group extensions in the context of homotopy theory.\footnote{In \cite{Whi}, a more general and slightly different definition is given. Our definition is the linear span of a \emph{crossed $G$-set} as in \cite{FY}*{Definition 4.2.1}.}
A \emph{$G$-crossed module} is a $G$-module
$$V=\oplus_{g\in G}V_g\qquad  \text{ such that }\qquad  h\cdot V_g=V_{hgh^{-1}}
$$
for any $h,g\in G$.
It was shown in \cite{FY}*{Theorem 4.2.2} that $G$-crossed modules form a braided monoidal category.
For a Hopf algebra $H$, the more general concept of \emph{$H$-Yetter--Drinfeld modules} \cite{Yet}*{Definition 3.6} recovers the definition of $G$-crossed modules used in this paper for the group algebra $H=\Bbbk G$. For finite-dimensional $H$, the category of Yetter--Drinfeld modules is equivalent to that of modules over Drinfeld's quantum double $\Drin H$ of \cite{Dri}, see \cite{Maj1}*{Proposition 7.1.6} and \cite{Maj5}.

The category of modules over $\Drin H$ is equivalent to the \emph{monoidal center} $\cZ(\cC)$ of the category $\cC=\Rep(H)$ of left $H$-modules, see e.g. \cite{Kas}*{Section XIII.5}. However, the braided monoidal category $\cZ(\cC)$ can be defined for any monoidal category $\cC$ \cites{Maj2,JS}.

\subsection{Deligne's Interpolation Category}
In \cite{Del}, Deligne constructs a class of Karoubian (i.e., idempotent complete) $\Bbbk$-linear symmetric monoidal categories $\uRep(S_t)$ depending on a parameter $t$ in the field $\Bbbk$ of characteristic zero. It is a well-known observation that every simple $S_n$-module appears as a direct summand of a tensor power $V_n^{\otimes k}$ of the $n$-dimensional regular representation $V_n$ of $S_n$ for some $k\geq1$. In other words, the category $\Rep(S_n)$ of finite-dimensional $S_n$-modules is the \emph{idempotent completion} (or \emph{Karoubi envelope}, \emph{Cauchy completion}, \cite{BD}) of a monoidal category generated by the single object $V_n$. The morphism spaces $\Hom(V_n^{\otimes k},V_n^{\otimes l})$ are given by the $S_n$-invariants of $V_n^{\otimes (k+l)}$ and can be described combinatorially using the partition algebras $P_k(n)$, see e.g. \cite{CO}*{Section~2} for details. To define $\uRep(S_t)$, $n$ is now replaced by a general parameter $t\in \Bbbk$.

In the generic case, if $t\notin \mZ_{\geq 0}$, then $\uRep(S_t)$ is a semisimple monoidal category. If $n\in \mZ_{\geq 1}$, then there exists a quotient functor
$$\cF_n\colon \uRep(S_n)\longrightarrow \Rep(S_n).$$
This way, $\uRep(S_t)$ may be seen as an interpolation category for the classical categories $\Rep(S_n)$.

\subsection{Dijkgraaf--Witten Invariants and TQFT}

We recall that there is a deep connection between invariants of knots, invariants of closed $3$-manifolds via surgery presentation, and representations of quantum groups which is manifest in the construction of $3$-dimensional TQFTs (also, $2$-dimensional rational conformal field theories). These links were explored by Witten in \cite{Wit} in connection with the \emph{Jones polynomial} of a knot \cite{Jon} and the quantized enveloping algebra of $\fr{sl}_2$ \cite{Dri}. Another interesting class of such $3$-dimensional TQFTs $Z^{\mathrm{DW}}_{G,\omega}$ was considered by Dijkgraaf--Witten \cite{DW} dependent on the input of a finite group $G$ and a $3$-cycle $\omega\in H^3(G,\Bbbk^\times)$. In the \emph{untwisted} case where $\omega$ is the trivial cocycle, the TQFT associates to any closed connected $3$-manifold $C$ the number
\begin{equation}
Z^{\mathrm{DW}}_G(C)=\frac{1}{|G|}\left|\Hom_{\mathrm{group}}(\pi_1(C),G)\right|=\frac{1}{|G|}\left|\Hom_{\mathrm{group}}(K_\cL,G)\right|,
\end{equation}
where $\Hom_{\mathrm{group}}$ denotes the set of group homomorphisms. Here, $\cL$ is any knot whose knot complement is a surgery presentation of $C$ and $K_\cL=\pi_1(C)$ is the \emph{knot group}. For a general ribbon link $L$, the quantity
\begin{equation}
\op{Inv}^{\mathrm{DW}}_G(L):=\left|\Hom_{\mathrm{group}}(K_\cL,G)\right|
\end{equation}
is a ribbon invariant. In fact, $Z^{\mathrm{DW}}_{G,\omega}$ extends to the structure of a fully extended $3$-dimensional TQFT for a general cocycle $\omega$. This TQFT was originally obtained using Chern--Simons theory and path integrals but can be constructed using a quasi-Hopf algebra $\Drin^\omega G$ generalizing the Drinfeld double $\Drin G=\Drin^{1}G$ \cites{DPR,AC,FQ,Fre2}.

\subsection{Summary of Results}

For $n\geq 0$, let $\mu$ be a partition of $n$, and $e_\rho$ be an idempotent in $Z\otimes M_k(\Bbbk)$ obtained, for example, from a simple representation $\rho$ of $Z=Z(\sigma)$, the centralizer of an element $\sigma$ of cycle type $\mu$ in $S_n$. Associated to this datum, we construct an interpolation object $\uW{\sigma,\rho}$, see \Cref{intobject-def} and \Cref{uWrho}.

\begin{reptheorem}{halfbraiding-prop}
The object $\uW{\sigma,\rho}$ gives an object of the monoidal center of $\uRep(S_t)$.
\end{reptheorem}

Up to isomorphism, $\uW{\sigma,\rho}$ only depends on the cycle type $\mu$ of $\sigma$. Hence, we define $\cD_t$ to be the (additive and) idempotent completion of the monoidal subcategory of $\cZ(\uRep(S_t))$ generated by objects of the form $\uW{\sigma,\rho}$ for all $n,\sigma,\rho$, see \Cref{Dtribbonsect}. We prove the following main theorems.

\begin{reptheorem}{ribbon-thm}
The category $\cD_t$ is a ribbon category.
Given an interpolation object $\uW{\sigma,\rho}$ in $\cZ(\uRep(S_t))$, the twist is given by
\begin{align}
\theta_{\uW{\sigma,\rho}}=(\sigma^{-1})^{\oplus k} e_\rho.
\end{align}
\end{reptheorem}

\begin{reptheorem}{maintheorem1}
The restriction of the functor $\cQ\colon \cZ(\uRep(S_n))\to \cZ(\Rep(S_n))$ to $\cD_n$ is essentially surjective and full on morphism spaces.
\end{reptheorem}

\Cref{maintheorem1} states that $\cD_t$ interpolates the category $\cZ(\Rep(S_n))$ of $S_n$-crossed modules, it implies that $\cQ$ induces an equivalence of the semisimplification of $\cZ(\uRep(S_n))$ (or of $\cD_n$) and $\cZ(\Rep(S_n))$ (\Cref{cor-semisimplification}). It remains an open question whether the monoidal center $\cZ(\uRep(S_t))$ is equivalent to $\cD_t$, cf. \Cref{Dtquestions}. If $t$ is a non-negative integer, then $\cD_t$ is not semisimple. It remains to be explored whether $\cD_t$, as $\uRep(S_t)$, is semisimple if $t$ is generic.

\Cref{ribbon-thm} implies that every object in $\cD_t$ provides polynomial invariants of (framed) ribbon links and ribbon knots. For $\mu,\rho$, the invariant of a ribbon link $\cL$ is denoted by  $\rP_{\mu,\rho}(\cL,\mathbf{t})$, where $\mathbf{t}$ is a free variable. If the twists are trivial, which we show to be the case in certain interesting situations, then these are invariants of links and knots. These ribbon link polynomials interpolate the untwisted Dijkgraaf--Witten invariants.

\begin{repcorollary}
{DWinterpolation}
For any be a ribbon link,
evaluating  the ribbon link polynomials $\rP_{\mu,\rho}(\cL,\mathbf{t})$ for a partition $\mu$ of $n$ at ${\mathbf{t}}=n$ recovers the untwisted Dijkgraaf--Witten invariants associated to $S_n$.
\end{repcorollary}

We compute examples of the new ribbon link polynomials in \Cref{invariantexamples}. For instance, for the left-handed trefoil we find in the case of the trivial character $\rho=\triv$ that
\begin{gather*}
    \rP_{(2), \triv}\left(\torusknotinvsmall{2}{3}~, \mathbf{t}\right)=(2\mathbf{t}-3) \frac{\mathbf{t}(\mathbf{t}-1)}{2}, \qquad 
    \rP_{(3), \triv}\left(\torusknotinvsmall{2}{3}~, \mathbf{t}\right)=(3\mathbf{t}-8) \frac{\mathbf{t}(\mathbf{t}-1)(\mathbf{t}-2)}{3}, \\
    \rP_{(4), \triv}\left(\torusknotinvsmall{2}{3}~, \mathbf{t}\right)=(2  \mathbf{t}^2-16  \mathbf{t}+37)\frac{\mathbf{t}(\mathbf{t}-1)(\mathbf{t}-2)(\mathbf{t}-3)}{4},\\
    \rP_{(2,2), \triv}\left(\torusknotinvsmall{2}{3}~, \mathbf{t}\right)=(4 \mathbf{t}^2-28 \mathbf{t}+49)\frac{\mathbf{t}(\mathbf{t}-1)(\mathbf{t}-2)(\mathbf{t}-3)}{8}.
\end{gather*}

The paper is structured as follows.
After recalling the necessary preliminaries in \Cref{prelims}, the main constructions and results are contained in \Cref{delignecenter-section}.
In \Cref{invariantsect}, we consider the ribbon link invariants obtained from $\cD_t$. Here, after recalling general material on Dijkgraaf--Witten theory and invariants obtained from ribbon categories, we clarify the relationship of the ribbon link polynomials obtained here to untwisted Dijkgraaf--Witten invariants in \Cref{untwistedrel}. 
\Cref{idempotent-appendix} considers the Karoubian envelope of the monoidal center.

\subsection{Acknowledgements}
The authors are grateful to
Pavel Etingof, Kobi Kremnizer, and  Victor Ostrik  for interesting discussions on the subject matter. In particular, V. Ostrik gave an important idea to start the construction of the interpolation objects. We further thank Sebastian Posur and two anonymous referees for their helpful comments that helped improve the exposition of the paper. The research of R.~L. was supported by the Simons foundation and is supported by a Nottingham Research Fellowship. 

\section{Preliminaries}\label{prelims}

In this section, we briefly discuss the necessary background material needed in this paper.

\subsection{The Monoidal Center} We recall basics about the center of a monoidal category of \cites{Maj2, JS}, see \cite{EGNO}*{Section 7.13} for a textbook exposition.

Let $\Bbbk$ be a field and $(\cC,\otimes)$ be a $\Bbbk$-linear monoidal category (that is, enriched over $\Bbbk$-vector spaces and additive, such that finite biproducts exist). The \emph{monoidal center} (or \emph{Drinfeld center})  $\cZ(\cC)$ of $\cC$  can be described as the monoidal category consisting of objects $(V,c)$, $V\in \cC$, and $c\colon V\otimes \ide_{\cC}\isomorph \ide_{\cC}\otimes V$ a natural isomorphism satisfying the $\otimes$-compatibility that
\begin{equation}\label{tensorcomp}
c_{X\otimes Y}=(\ide_X\otimes c_{Y})(c_X\otimes \ide_Y),
\end{equation}
where associativity isomorphisms are omitted in the notation. We call such an isomorphism $c$ a \emph{half-braiding}. If $\cC$ is rigid, a half-braiding is, equivalently, a natural transformation $c\colon V\otimes \ide_{\cC}\to \ide_{\cC}\otimes V$ satisfying \Cref{tensorcomp} such that the half-braiding with the tensor unit $c_I$ corresponds to the identity on $V$. Note that there exists a faithful strict monoidal functor $\rF\colon \cZ(\cC)\to \cC$ which sends a pair $(V,c)$ to the object $V$.

The center $\cZ(\cC)$ is a braided monoidal $\Bbbk$-linear category, with braiding $\Psi$ given by
\begin{align}\label{centerbraiding}
\Psi_{(V,c),(W,d)}=c_W.
\end{align}

\subsection{Ribbon Categories}\label{ribbon-sect}

A \emph{pivotal category} is a monoidal category $\cC$ (with unit $I$) together with an assignment $X\mapsto X^*$ such that $X^*$ is both a left and right dual of $X$, i.e., such that the left and right duality functors coincide (see \cite{TV}*{Section 1.7}). We denote the natural isomorphisms of left and right duality by
\begin{gather}
\ev^l_X \colon X^*\otimes X\to I, \qquad \coev^l_X\colon I\to X\otimes X^*, \\
\ev^r_X \colon X\otimes X^*\to I, \qquad \coev^r_X\colon I\to X^*\otimes X.
\end{gather}
A \emph{strict pivotal functor} is a strict monoidal functor which strictly preserves $X^*$, $\ev^l_X$, and $\ev^r_X$.

The monoidal center $\cZ(\cC)$ is pivotal provided that $\cC$ is pivotal \cite{TV}*{Section 5.2.2} with the following structure: The dual $(V^*,c^*)$ of an object $(V,c)$ in $\cZ(\cR_t)$ is defined on the dual $V^*$ of $V$ in $\cC$, where the dual half-braiding $c^*_X\colon V^*\otimes X\to X\otimes V^*$ is given by
\begin{align}\label{centerdual}
c^*_X = (\ev^l_V\otimes \ide_{X\otimes V^*})(\ide_{V^*}\otimes c_X^{-1}\otimes \ide_{V^*})(\ide_{V^*\otimes X}\otimes \coev^l_V),
\end{align}
for any object $X$ of $\cC$.

In a pivotal category, we can define left and right \emph{traces}, and a pivotal category is called \emph{spherical} if they coincide, that is,
\begin{align}
    \ev^l_X(\ide_{X^*}\otimes f)\coev^r_X
    = \ev^r_X(f\otimes \ide_{X^*})\coev^l_X
\end{align}
for any endomorphism $f$ of $X$ in $\cC$. The resulting endomorphism of $I$ is called $\tr(f)$, the \emph{trace} of $f$, and can be regarded as an element in $\Bbbk$. As a special case, we obtain the \emph{dimension} $\dim(X):=\tr(\ide_X)$ for any $X$ in $\cC$.

The categories considered in this article will be spherical. For instance, $\cZ(\cC)$ is a spherical category if $\cC$ is a spherical category, and the forgetful functor $\rF\colon \cZ(\cC)\to \cC$ is strictly pivotal \cite{TV}*{Section 5.2.2}.

We recall the concept of a \emph{ribbon category}, see e.g. \cite{TV}*{Section 3.3.2}. Given an object $X$ in a braided pivotal category $\cC$, we can define two endomorphisms of $X$, the \emph{left twist} $\theta^l_X$, and the \emph{right twist} $\theta^r_X$, by
\begin{gather*}
\theta^l_X=(\ev^l_X\otimes \ide_X)(\ide_{X^*}\otimes \Psi_{X,X})(\coev^r_X\otimes \ide_X),\\
\theta^r_X=(\ide_X\otimes \coev^r_X)(\Psi_{X,X}\otimes \ide_{X^*})(\ide_X\otimes \coev^l_X).
\end{gather*}
The category $\cC$ is a ribbon category if the left twist $\theta^l_X$ and the right twists $\theta^r_X$ are equal for every object $X$ of $\cC$. In this case, we write $\theta_X:=\theta_X^l$. Every ribbon category is spherical (see \cite{TV}*{Corollary 3.4}).

\subsection{The Monoidal Center of \texorpdfstring{$\Rep(S_n)$}{Rep(Sn)}}\label{Sncenter}

In this section, we recall basic facts about the monoidal center of the  category $\Rep(S_n)$ of finite-dimensional $S_n$-modules over the field $\Bbbk$. It is well known that there are equivalences\footnote{The equivalences hold for any finite-dimensional Hopf algebra $H$ instead of $\Bbbk S_n$, see e.g. \cites{Kas,Maj1}.} of braided monoidal categories
\begin{align}
\cZ(\Rep(S_n))\simeq \lYD{\Bbbk S_n}\simeq\Rep (\Drin S_n).
\end{align}
see e.g. \cites{Kas,Maj1}, where the second and third category will be explained in the following.

The category $\lYD{\Bbbk S_n}$ is that of finite-dimensional \emph{Yetter--Drinfeld modules} over $\Bbbk S_n$, also called \emph{$S_n$-crossed modules}. An $S_n$-crossed module is a $S_n$-graded $S_n$-module
$$V=\bigoplus_{g\in S_n} V_g,
\qquad \text{such that }\qquad hV_g=V_{hgh^{-1}}
\qquad \text{for all}\qquad g,h\in S_n\ .
$$
A morphism of $S_n$-crossed modules is a morphism of $S_n$-modules which preserves the grading.

The category $\Rep(\Drin S_n)$ consists of all finite-dimensional modules over the \emph{Drinfeld double} (or \emph{quantum double}) $\Drin S_n$ of the group algebra $\Bbbk S_n$. We note that the quasi-triangular Hopf algebra $\Drin S_n$ has dimension $|G|^2$.  Note that as $\Bbbk$ is a field of characteristic zero, $\cZ(\Rep(S_n))$ is a semisimple category \cite{DPR}*{Section 2.2}.

An $S_n$-crossed module $V$ defines an object $(V,c)$ in $\cZ(\Rep(S_n))$, with the half-braiding given by
\begin{align}
&c_W\colon V\otimes W\longrightarrow W\otimes V, &c_W(v\otimes w)=(g\cdot  w)\otimes v,&&\text{for } v\in V_g, w\in W,
\end{align}
where now $W$ can be any $S_n$-module.

\begin{definition}
Let $\mu\vdash n$ be a partition of $n$.
We define the $S_n$-crossed module $V^{\mu}$ as
\begin{align*}
V^{\mu}
= \mu^{S_n}\otimes \Bbbk S_n
= \bigoplus \Set{ \Bbbk g\otimes h \mid g,h\in S_n,\text{ such that the cycle type of $g$ is $\mu$}}
\ ,
\end{align*}
where $\mu^{S_n}$ denotes the $S_n$-permutation representation on the subspace of the group algebra spanned by elements of cycle type $\mu$. The $S_n$-action and an $S_n$-grading  on $V^\mu$ are given by
\begin{align}
\sigma\cdot (g\otimes h)&=\sigma g \sigma^{-1}\otimes \sigma h, & g\otimes h \in (V^\mu)_g, &&\text{for } g\otimes h\in V^{\mu}, \sigma \in S_n,
\end{align}
That is, $V^\mu$ is the tensor product of the permutation module $\mu^{S_n}$ with the regular module $\Bbbk S_n^{\reg}$.
\end{definition}

\begin{lemma}\label{regulardec-lemma}
The regular $\Drin S_n$-module $V^{\reg}=\Drin S_n$ decomposes as a direct sum
\begin{align*}
V^{\reg}\cong \bigoplus_{\mu\vdash n}V^{\mu}
\end{align*}
of $\Drin S_n$-modules. Here, $\mu$ ranges over all partitions of $n$.
\end{lemma}

In this paper, we will need to evaluate the half-braiding of $V^{\mu}$ on the $\Bbbk S_n$-module
$$
V_n = \Bbbk v_1 \oplus\dots\oplus \Bbbk v_n
\ ,
$$
which is the $n$-dimensional \emph{standard module}, with $\sigma\cdot v_i=v_{\sigma(i)}$ for $\sigma\in S_n$. Clearly, the half-braiding $c_{V_n}$ evaluated on the given basis is
\begin{align}\label{ch}
c_{V_n}((g\otimes h)\otimes v_i)=v_{g(i)}\otimes (g\otimes h), &&\text{for } g\otimes h\in V^\mu, v_i\in V_n.
\end{align}

We finish this section by recalling an explicit description of the \emph{simple} $S_n$-crossed modules. To this end, let $\mu=(\mu_1,\dots,\mu_k)$ be a partition of $n$ and let us denote the multiplicity of any part $1\leq i\leq n$ in $\mu$ by $n_i\geq 0$, that is, the partition $\mu$ can be described as
$$n =\mu_1+\dots+\mu_k= \sum_{1\leq i\leq n} n_i i \ .$$
Let $\sigma\in S_n$ be an element of cycle type $\mu$, that is, it can be written as a product of disjoint cycles in $S_n$ of lengths $\mu_1,\dots,\mu_k$. The centralizer $Z(\sigma)$ of $\sigma$ in $S_n$ is then isomorphic to the wreath product $\prod_i \mZ_i \wr S_{n_i}\cong \prod_{i: n_i>0} \mZ_i \wr S_{n_i}$.

For a $k\geq 1$, Let $V=\Bbbk^k$ be a simple $Z(\sigma)$-module with action $\rho\colon Z(\sigma)\to M_k(\Bbbk)$. Then we can consider the induced module $W=\Bbbk S_n \otimes_{\Bbbk Z(\sigma) } V$. More explicitly, let us choose representatives $(h_i)_i$ for the left cosets $S_n/Z(\sigma)$, i.e., $(h_i)_i$ are elements from $S_n$, one in each left coset of $Z(\sigma)$ in $S_n$. Then $W=\oplus_i \Bbbk h_i\otimes V$ as a vector space, and $W$ obtains a $S_n$-grading by assigning the degree $h_i\sigma h_i\inv$ to the subspace $h_i\otimes V$ such that the non-zero graded components are all isomorphic to $V$ as vector spaces and correspond exactly to the conjugacy classes of $\sigma$.

For each $h_i$ and each $h\in S_n$, we can write $hh_i=h_jz$ for some $h_j$ and some $z\in Z(\sigma)$, thus we can define an action of $h$ on $h_i\otimes V$ by setting
\begin{align}\label{inducedaction} 
h\cdot (h_i\otimes v) &= h_j \otimes \rho(z)v, &&\text{for } v\in V.
\end{align}
It can be verified that this defines an action of $S_n$ on $W$ and the degree of $h\cdot(h_i\otimes v)$ is $h_j\sigma h_j\inv=h (h_i \sigma h_i\inv) h\inv$, so we have constructed an $S_n$-crossed module.

It turns out that up to isomorphism the centralizer group $Z(\sigma)$ and the $S_n$-crossed module $W$ only depend on the cycle type $\mu$ of $\sigma$ (or equivalently, the conjugacy class of $\sigma$). We can hence denote the group and the $S_n$-crossed module constructed by $Z(\mu)$ and $W_{\mu,\rho}$, respectively.

\begin{theorem}[\cite{DPR}]\label{Sncenter-thm} For every partition $\mu$ of $n$ and every simple representation $(V,\rho)$ of $Z(\mu)$ , $W_{\mu,\rho}$ is a simple $S_n$-crossed module. Every irreducible $S_n$-crossed module is isomorphic to such a module for suitable $\mu,\rho$. Two such modules are isomorphic if and only if they are constructed using the same partition and isomorphic modules of the respective centralizer groups.
\end{theorem}

\subsection{Deligne's Category \texorpdfstring{$\uRep(S_t)$}{Rep(St)}}\label{deligne-sect}

We now fix some notations and introduce the interpolation category $\uRep(S_t)$ for $t\in \Bbbk$. We refer the reader to \cites{Del,CO} for more details. Recall that $\Bbbk$ is a field of characteristic zero and let $P_{n,m}$ denote the set of all partitions of the set $\Set{1,\ldots,n}\cup \Set{1',\ldots,m'}$ with $n+m$ elements, for all $n,m\geq 0$. Given a partition $\lambda$, we denote by $|\lambda|$ the number of its connected components. We utilize pictures to denote such partitions, for example,
\[
\lambda=
\tp{3,1,11,0,2,0,12,13,0,4}\quad \in \quad P_{4,3}
\]
denotes the partition $\lambda=\Set{\Set{1,3,1'},\Set{2},\Set{4},\Set{2',3'}}$. The composition of such morphisms is given by concatenation, reading from the top to bottom. All vertices but those in the first or last row are removed from each connected component of the resulting graph and a scalar factor $t$ is introduced for each component which is removed completely in the process. For example,
\begin{align*}
\left(
\tp{1,11, 12,0, 2,0,3}\right)
\circ
\left(\tp{3,1,11,0,2,0,12,13,0,4}\right)
=
\tp{3,1,11,21,22,0,2,0,12,13,0,4}
=
t\cdot \left(\tp{3,1,11,12,0,2,0,4}\right)\quad \in \quad P_{4,2}.
\end{align*}

Define $\cR^0_t=\uRep_0(S_t)$ as the category with
\begin{itemize}
\item objects $[n]$, for $n\in \mN_0$;
\item morphisms $\Hom([n],[m])=\Bbbk P_{n,m}$, the vector space with a basis indexed by the partitions in $P_{n,m}$;
\item composition defined on basis elements as above.
\end{itemize}
Note that for any $n\in\mN_0$, the identity endomorphism on $n$
corresponds to the partition $\{\{1,1'\},\dots,\{n,n'\}\}$. We denote it by $1_n$.

We define $\cR^1_t=\uRep_1(S_t)$ to be the additive closure of $\cR^0_t$. We use the following model for $\cR^1_t$. Objects are thought of as order lists, denoted as formal direct sums of objects in $\cR^0_t$. That is, a typical object is $[n_1]\oplus\ldots\oplus [n_k]$. For convenience of notation later, we define a vector space tensor action by
$$[n]\otimes \Bbbk^l=\underbrace{[n]\oplus \ldots \oplus [n]}_{\text{$l$ summands}}=[n]^{\oplus l},$$
for $l\geq 0$, which is regarded as an ordered sum, where we use the convention that $[n]\otimes \Bbbk^0=[n]^{\oplus 0}=[0]$, the tensor unit of $\cR^0_t$. Given a matrix $A\in M_{k\times l}(\Bbbk)$ and $f\colon [n]\to [m]$ a morphism in $\cR^0_t$, we obtain a morphism
\begin{align}f\otimes A\colon [n]^{\oplus k}\longrightarrow [n]^{\oplus l},\label{matrix-notation}\end{align}
where the $(i,j)$-th entry is the morphism $A_{ij}f$ in $\cR^0_t$.

Next, we denote the idempotent completion (also called \emph{Karoubian envelope}) of $\cR^1_t=\uRep_1(S_t)$ by $\cR_t=\uRep(S_t)$ (cf. Appendix \ref{idempotent-appendix}).
The category $\cR_t$ is a monoidal category, cf. \cite{CO}*{Proposition 2.2}. It is symmetric monoidal using the symmetry $\Psi$ which crosses strands \cite{CO}*{p.~1340}. Duals in $\cR_t$ are obtained from duals in $\cR^0_t$. That is, $[n]$ is self-dual (i.e. $[n]^*=[n]$) with the pairing $\ev\colon [n]^*\otimes [n]\to [0]$ described explicitly in \cite{CO}*{Section 2.2}. The object $[n]\otimes \Bbbk^l$ is self-dual with pairing given by the pairing of $[n]$ tensored with the standard pairing of $\Bbbk^l$, defined on the standard basis by $\ev(e_i,e_j)=\delta_{i,j}$. A morphism $f\otimes A$ is dualized to $f^*\otimes A^t$, where $f^*$ is the dual morphism in $\cR^0_t$. Given an idempotent $e$ in $\cR_t$, the dual object $([n],e)^*$ is $([n]^*, e^*)$, where $e^*$ is the dual morphism to $e$.

\begin{definition}[The partial order $>$, the elements $x_\lambda$]\label{ordering}
Recall that there is a partial order on partitions $P_{n,m}$. For $\lambda,\mu\in P_{n,m}$ we write $\mu\geq\lambda$ if $\mu$  is \emph{coarser} than $\lambda$, i.e. if two indices are in the same part of the partition $\lambda$, then they are in the same part in $\mu$. We write $\mu>\lambda$ if $\mu\geq\lambda$ and $\mu\neq\lambda$.

Recall that for any $\lambda\in P_{n,m}$, we have the recursively defined \cite{CO}*{Equation (2.1)} elements
\begin{align}
x_\lambda=\lambda-\sum_{\mu>\lambda}x_{\mu}.
\end{align}
\end{definition}

We observe that $\cR_t$ (as well as $\cR^0_t, \cR^1_t$) are ribbon categories, and hence spherical and pivotal (see Section \ref{ribbon-sect} for these concepts). Regarding the pivotal structure we note that for any idempotent $e\colon [n]\to [n]$, $([n],e^*)$ is left dual to $([n],e)$ as explained above, and right dual by use of the symmetry. Hence $([n],e)^{**}=([n],e)$ and we can use identity morphisms to give a twist structure making $\cR_t$ a ribbon category (see e.g. \cite{TV}*{Section 3.3.3}). In this category we have
\begin{align}
\ev_X:=\ev_X^l&=\ev_X^r, &\coev_X:=\coev_X^l&=\coev_X^r,
\end{align}
for any object $X=([n],e)$.
Thus, the \emph{trace} of an endomorphism $f$ of $X$ in $\cR_t$ is simply
$$\tr(f)=\ev_X\Psi_{X,X^*} (f\otimes \ide_{X^*})\coev_{X},$$
cf. \cite{CO}*{Section 2.3}. For instance, the trace of an endomorphism $f$ of the object $[4]$ corresponds to the following diagram:
\[ \tikz[looseness=1]{
\node[align=center, text width=20pt, draw=black](f){$f$};
\node[align=center, text width=20pt,right=of f](id){\phantom{$f$}};
\foreach \a in {60,80,100,120}
\draw (f.\a) to[out=90,in=90] (id.{180-\a}) -- (id.{\a-180}) to[out=-90,in=-90] (f.-\a);
} 
\]
We note that $\dim([n])=\tr(1_n)=t^n$.

Further, recall that $\cR_t$ is semisimple for $t$ \emph{generic}, i.e., for $t$ \emph{not} an integer larger or equal to zero \cite{Del}*{Th\'eor\`eme 2.18}.

We collect here some notation used in the paper concerning $\cR_t$. The identity on the object $[n]$ is denoted by $\ide_{[n]}=1_n$. We read compositions of morphisms in $\cR_t$ from right to left. We often simply write $f\circ g=fg$. We use matrix notation, as made precise above, to denote morphisms in $\cR^1_t$ or $\cR_t$. In particular, $f^{\oplus s}$ denotes the diagonal $s\times s$-matrix with $f$ on the diagonal.

\subsection{The Interpolation Functor \texorpdfstring{$\cF_n$}{Fn}}\label{functor-sect}

In this section, we restrict to the case where $t=n$ is a non-negative integer. Deligne introduced the \emph{interpolation functor}
$$\cF_n \colon \uRep(S_n)\longrightarrow\Rep(S_n),$$
see \cite{Del}*{\S 6}, \cite{CO}*{Section 3.4} where, by convention, $S_0$ is the trivial group. The monoidal functor $\cF_n$ is essentially surjective on objects, and full on morphism spaces. 
An explicit description of this functor is given as follows. 

Recall that $V_n$ is the $n$-dimensional standard module of $S_n$. For any $d\geq 0$, $\cF_n([d])=V_n^{\otimes d}$ (where $V_0^{\otimes 0}=\Bbbk$), and $\cF$ is compatible with direct sums.

Let us pick a basis $v_1,\dots,v_n$ of $V_n$ and for any $d\geq 1$, let us denote the tensor product $v_{i_1}\otimes\dots\otimes v_{i_d}$ in $V_n^{\otimes d}$ by $v_{\mathbf{i}}$ for any $d$-tuple $\textbf{i}=(i_1,\dots,i_d)$ in $\Set{1,\dots,n}^d$. Then there is a linear map
\begin{align*}
f\colon \Hom_{\cR_t}([d],[e])\longrightarrow \Hom_{S_n}(V_n^{\otimes d},V_n^{\otimes e})
\end{align*}
sending any partition $\pi\in P_{d,e}$ to the function defined by
\begin{align}\label{piimage}
f(\pi)(v_{\mathbf{i}})&=\sum_{\mathbf{j}}f(\pi)_{\mathbf{j}}^{\mathbf{i}}v_{\mathbf{j}},
\end{align}
for all $d$-tuples $\mathbf{i}$ and all $e$-tuples $\mathbf{j}$, where the coefficient $f(\pi)_{\mathbf{j}}^{\mathbf{i}}$ equals $1$ if the partition $\pi$ of the set $\{1,...,d,1',...,e'\}$ equals or refines the partition which is given by the equality of the corresponding indices $\mathbf{i},\mathbf{j}$, and $f(\pi)_{\mathbf{j}}^{\mathbf{i}}=0$ otherwise (see e.g. \cite{CO}*{Section 2.1} for more details).

We note that for a morphism $x_\pi$ as in Definition \ref{ordering}, $f$ can be computed more directly,
\begin{align}\label{ximage}
f(x_\pi)(v_{\mathbf{i}})&=\sum_{\mathbf{j}}f(x_\pi)_{\mathbf{j}}^{\mathbf{i}}v_{\mathbf{j}},
\end{align}
where the coefficient $f(x_\pi)_{\mathbf{j}}^{\mathbf{i}}$ equals $1$ if  the partition $\pi$ of the set $\{1,...,d,1',...,e'\}$ equals the partition which is given by the equality of corresponding indices $\mathbf{i},\mathbf{j}$. Otherwise, $f(x_\pi)_{\mathbf{j}}^{\mathbf{i}}=0$.

\begin{remark}\label{rem-f-x-g-n} In particular, for $\pi=1_n\in P_{n,n}$, $f(x_{1_n})$ is the endomorphism of $V_n^{\otimes n}$ which sends $v_{\mathbf{i}}$ to $0$ if two indices of the tuple $\mathbf{i}$ coincide, while it is the identity map on $v_{\mathbf{i}}$ if $\mathbf{i}$ has $n$ distinct indices. We can hence identify the image of $f(x_{1_n})$ with $\Bbbk S_n$ as an $S_n$-module, where we identify $v_{\mathbf{i}}$ with the permutation $k\mapsto i_k$ for $1\leq k\leq n$ and where $g\cdot v_{\mathbf{i}} = v_{g(i_1),\dots,g(i_n)}$ for $g\in S_n$.

Similarly, for any permutation $g\in S_n$, $f(x_g)$ sends $v_\mathbf{i}$ to $0$ if two indices in $\mathbf{i}$ coincide, and sends $v_{(i_1,\dots,i_n)}$ to $v_{(i_{g\inv(1)},\dots,i_{g\inv(n)})}$ for distinct indices $i_1,\dots,i_n$. That is, $f(x_g)$ corresponds to the right multiplication by $g\inv$ on $\Bbbk S_n$, reading the composition of permutations from right to left.
\end{remark}

It turns out that $f$ is functorial with respect to the composition, and hence can be used to define the functor $\cF_n$, where the assignment on objects sends $([d],e)$ to the image of the idempotent $f(e)$.
Given an $S_n$-module $V$, we may call an object $([n],e)$ in $\cR_t$ such that $\cF_n([n],e)\cong V$ a \emph{lift} of $V$.
The functor $\cF_n$ has an interpretation in terms of semisimplification of $\cR_t$, see \Cref{functorcenter}.

\subsection{Generating Morphisms for \texorpdfstring{$\uRep(S_t)$}{Rep(St)}}

The monoidal category $\cR^0_t=\uRep_0(S_t)$ is generated by the object $[1]$ under tensor products, and hence $\cR^1_t$ is generated by $[1]$ under direct tensor products and direct sums. In this section, we determine a set of generating morphisms for these categories. We denote by $\Hom_0,\Hom_1$, $\Hom$ the set of morphisms in $\cR^0_t,\cR^1_t,\cR_t$, respectively

\begin{definition}[The set $M$] We consider the following morphisms in $\cR^0_t$:
\begin{gather*}
\pi^* = \tpart{1,0}{}
  \in\Hom_0([1],[0])\ ,
\quad
\pi_* = \tpart{0,1}{}
  \in\Hom_0([0],[1])\ ,
\quad
\pi_H = \tpart{2,2}{\tpid{1,2} \draw (1,1.5)--(2,1.5);}
  \ ,
\quad
\pi_X = \tpart{2,2}{\draw (1,1)--(2,2) (1,2)--(2,1);}
  \in\Hom_0([2],[2])\ .
\end{gather*}
Now let $M$ be the set $\{\pi_*,\pi^*,\pi_H,\pi_X\}$.
\end{definition}

\begin{proposition}\label{generatorprop} The set $M$ generates $\Hom_0$ under composition, tensor product, and $\Bbbk$-linear combination of morphisms. 
\end{proposition}

\begin{proof} First, we fix $n\geq 1$. Now any permutation $\sigma\in S_n$ defines a morphism $\pi_\sigma$ in $\Hom_0([n],[n])$, namely the partition $\{\{1,\sigma(1)'\},\dots,\{n,\sigma(n)'\}\}$. To show that all these morphisms are generated by $M$, it is enough to consider the case where $\sigma$ is the identity permutation or a transposition, because $S_n$ is generated by transpositions, and the composition of $\pi_{\sigma_1}$ and $\pi_{\sigma_2}$ in the category $\cR^0_t$ yields $\pi_{\sigma_1\sigma_2}$ for permutations $\sigma_1,\sigma_2\in S_n$. But
\[ \pi_{\sigma=\ide} = 1_n
\textnormal{ and }
\pi_{\sigma=(k\,k+1)} = 1_{k-1}\otimes \pi_X\otimes 1_{n-k-1}
\]
for any $1\leq k<n$, so indeed, $M$ generates all morphisms in $\cR^0_t$ which come from permutations.

Second, we fix $t_1,t_2\geq 0$ and we let $\pi_{t_1,t_2}$ be the partition $\{\{1,\dots,t_1,1',\dots,t_2'\}\}$. Now
\[
\pi_{0,0} = 1_0 \ ,\quad
\pi_{1,1} = 1 \ ,\quad
\pi_{t+1,t+1} = (1_{t-1}\otimes \pi_H) (\pi_{t,t}\otimes1)
\]
for all $t>1$. Hence $\pi_{t,t}$ is generated by $M$ for all $t\geq0$.

Now if $t_1>t_2$, then $\pi_{t_1,t_2}=(1_{t_2}\otimes (\pi^*)^{\otimes(t_1-t_2)})\pi_{t_1,t_2}$, and a similar equality involving $\pi_*$ holds if $t_2>t_1$, proving that $M$ generates $\pi_{t_1,t_2}$ for all $t_1,t_2\geq0$.

To prove the assertion by induction, we note that $\Hom_0([0],[0])=\Bbbk 1_0$ is generated by $M$. So let us consider $m,n\geq 0$, one of them positive, assuming that $\Hom_0([m'],[n'])$ is generated by $M$ for all $0\leq m'\leq m,0\leq n'\leq n$ as long as $m'<m$ or $n'<n$. We consider any partition $\pi$ in $\Hom_0([m],[n])$. Let us pick any connected component of $\pi$. By pre- and post-composing with suitable permutations we can reorder the vertices such that the selected connected component consists exactly of the first $t_1$ upper vertices and the first $t_2$ lower vertices, where $t_1$ or $t_2$ may be $0$, but not both, since the component is non-empty. This modified partition decomposes as a tensor product of a partition $\pi_{t_1,t_2}$ corresponding to the selected connected component and a remainder. By the induction hypothesis, the remainder is generated by $M$, and as we showed previously, $\pi_{t_1,t_2}$ and also permutations are generated by $M$. Hence $\pi$ is generated by $M$.

As any morphism in $\cR^0_t$ is a linear combination of partitions, this proves the assertion.
\end{proof}

To illustrate the notation from the proof with an example, take $t_1=4,t_2=3$, then
\[
\pi_{4,3} = \tpart{4,4,3,3}{\tpid{1,...,4} \tpid[2]{1,...,3} \tpid[3]{1,...,3} \draw (1,1.5) -- (4,1.5) (1,3.5)--(3,3.5);}
\ .
\]
We note that this set of generating morphisms can alternatively be derived from the universal property of $\uRep(S_t)$ proved in \cite{Del}*{Proposition 8.3}, which exhibits $\uRep(S_t)$ as the universal symmetric monoidal category with a Frobenius algebra object.

\begin{corollary} The morphisms in $\Hom_1$, $\Hom$ are matrices of morphisms generated by $M$ under composition, tensor product, and $\Bbbk$-linear combination.
\end{corollary}

\section{The Monoidal Center of Deligne's Category}\label{delignecenter-section}

The goal of this section is to give explicit descriptions of a large class of objects in the monoidal center $\cZ(\cR_t)$ of Deligne's category $\cR_t=\uRep(S_t)$. These objects generate a ribbon category $\cD_t$ with a full and essentially surjective functor to $\cZ(\Rep(S_n))$ if $t=n$.

Let $(V,c)$ be an object in $\cZ(\cR_t)$, where $V=([n],e)$ is an object in $\cR_t$ and $c\colon V\otimes \ide_{\cR_t}\isomorph \ide_{\cR_t}\otimes V$ is a half-braiding. We observe that, as all morphisms in $\cR_t$ are also morphisms in $\cR_t^1$, the half-braiding $c$ is determined by the isomorphism $c_1:=c_{[1]}$.
Indeed, given $c_1$, we obtain $c_m=c_{[m]}$ inductively by setting $c_{[0]}=\ide_V$, and
\begin{align}\label{inducedc}
c_{m+1}=(\ide_V\otimes c_1)(c_{m}\otimes 1).
\end{align}

Equivalently, $c$ is determined by the isomorphism
\begin{align}
d_1:=\Psi_{[1],[n]}c_{1}\colon V\otimes [1]\isomorph V\otimes [1].\label{d1}
\end{align}
Moreover, consider the isomorphism
$c_2\colon V\otimes [2]\isomorph [2]\otimes V,$
or, equivalently, the isomorphism
\begin{align} d_2:= (\ide_{V}\otimes \Psi_{[1],[1]})(d_1\otimes 1)(\ide_{V}\otimes \Psi_{[1],[1]})(d_1\otimes 1).\label{d2}
\end{align}

Note that the morphisms $c_1,c_2$ and $d_1,d_2$, respectively, carry the same information, we introduce the latter to simplify certain computations (such as checking the naturality conditions \eqref{centercond1}--\eqref{centercond2} below) in later proofs.

\subsection{Conditions for Objects in  \texorpdfstring{$\cZ(\uRep(S_t))$}{Z(Rep(St))}}\label{centerconds}
In this section, we present commutation relations which characterize exactly those $d_1, d_2$ (as above) which determine an object in $\cZ(\cR_t)$.

\begin{proposition}\label{centerconditions} Let $W=([n]^{\oplus k},e)$ be an object in $\cR_t$, and let $d_1\colon W\otimes [1]\to W\otimes [1]$ be a morphism. Then $d_1$ determines a half-braiding $c$ through Equations (\ref{inducedc})--(\ref{d1}), and hence an object $(W,c)$ in $\cZ(\cR_t)$, if and only if \begin{gather}
 d_1 (\ide_{W}\otimes\pi_*) = \ide_W\otimes \pi_* \ ,
\quad
 (\ide_W\otimes\pi^*)  d_1 = \ide_W\otimes \pi^*\ ,\label{centercond1}
\\
 d_2 (\ide_W\otimes\pi_X) = (\ide_W\otimes\pi_X) d_2\ ,
\quad
 d_2 (\ide_W\otimes\pi_H) = (\ide_W\otimes\pi_H) d_2\ .\label{centercond2}
\end{gather}

\end{proposition}
\begin{proof}
The objects of the category $\cR_t^1$ are generated under tensor products and direct sums by the object $[1]$. Therefore, as detailed in Equations (\ref{inducedc})--(\ref{d1}), $d_1$ determines $c_{[n]}$. Hence, using Proposition \ref{idempotentcenter-prop}, $d_1$ determines a half braiding $c$ if we can show that the resulting morphisms $c_{[n]}$ are natural in the second component. That is, for any morphism $f\colon [n]\to [m]$, we require
$$c_{[m]}(\ide_W\otimes f)=(f\otimes\ide_W)c_{[n]}\ .$$
Recall that $\cR_t$ is a pivotal category so that left and right dual exist (and coincide). This implies, by  \Cref{centerdual}, that if $c$ is natural, then it is automatically invertible.

It is sufficient to check that $c$ is natural with respect to generating morphisms. Using Proposition \ref{generatorprop}, this implies that Equations (\ref{centercond1})--(\ref{centercond2}) are necessary and sufficient for $c$ to define a half-braiding. 
Also, note that by  \Cref{tensorcomp}, $c_0=\ide_W$.
\end{proof}

\subsection{Lifts of Induced Representations}\label{liftinduced-sect}

We recall from Section \ref{Sncenter} that for any $n\geq 0$, the Drinfeld center of $\Rep(S_n)$ is the category of $S_n$-crossed modules. We also recall that if
\begin{itemize}
\item $\mu$ is a partition of $n$,
\item $\sigma\in S_n$ is an element of cycle type $\mu$, and
\item $Z=Z(\sigma)$ is the centralizer subgroup of $\sigma$ in $S_n$,
\end{itemize}
If $\rho\colon \Bbbk Z\to M_k(\Bbbk)$ determines an irreducible representation $V$ of $Z$,
then this data defines a simple object $W=W_{\mu,\rho}$ in $\cZ(\Rep(S_n))$ which is independent of and unique up to the choice of $\sigma$ and the isomorphism class of $V$, as explained in Theorem \ref{Sncenter-thm}. This yields a parametrization of all simple objects in this category. 
As an $S_n$-representation, $W_{\mu,\rho}=\Ind_Z^{S_n}(V)$.

As a first step towards lifting $W_{\mu,\rho}$ to an object in $\cZ(\cR_t)$, we now define objects in $\cR_t$ which are mapped to the induced $S_n$-representation $W_{\mu,\rho}$ under the functor $\cF_n$ from Section \ref{functor-sect}.


For the remainder of Section \ref{delignecenter-section}, let us fix $n,\mu,\sigma,Z$ as above and consider, more generally, $k\geq 1$ and $\rho\colon \Bbbk Z\to M_k(\Bbbk)$, a linear map satisfying
\begin{align}\label{eq-rho}
 \sum_{z_1 z_2=z} \rho(z_1)\rho(z_2)=|Z| \rho(z),
 &&\textnormal{ for all } z\in Z.
\end{align}
By the following lemma, the maps $\rho$ satisfying \Cref{eq-rho} are in 1-to-1 correspondence with the idempotents in the semisimple algebra $\Bbbk Z\otimes M_k(\Bbbk)$, that is, with the submodules of the left (or right) regular representation of this algebra.

\begin{lemma}\label{lem-rho-idempotent} \Cref{eq-rho} holds if and only if $|Z|\inv\sum_{z\in Z} z\otimes \rho(z)$ is an idempotent in $\Bbbk Z\otimes M_k(\Bbbk)$.
\end{lemma}

\begin{proof} Squaring $|Z|\inv \sum_{z\in Z} z\otimes \rho(z)$ we get
$|Z|\inv \sum_{z\in Z} z\otimes |Z|\inv \sum_{z_1 z_2=z} \rho(z_1)\rho(z_2)$.
\end{proof}

\begin{example}\label{ex-rho-irrep} Let $\rho$ be the action of $Z$ on a $k$-dimensional module. Then $\rho$ satisfies \Cref{eq-rho}.
\end{example}

\begin{example} Let $\chi$ be an irreducible character of $Z$ and define $\rho(z):=\chi(e) \chi(z\inv)\in\Bbbk$ for all $z\in Z$. Then $|Z|\inv\sum_{z\in Z} z\otimes \rho(z)$ is the $\chi$-isotypical projector, an idempotent, and $\rho$ satisfies \Cref{eq-rho}. Similarly, if $\chi$ is the sum of pairwise distinct (and hence, orthogonal) irreducible characters of $Z$, then $\rho$ satisfies \Cref{eq-rho}.
\end{example}

\begin{example} For the irreducible $Z$-module $V$, as $\Bbbk Z$ is semisimple, there is a primitive idempotent $e_V$ in $\Bbbk Z$ such that $V$ is isomorphic to $\Bbbk Z e_V$ as $Z$-module. We can write $e_V=\sum_{z\in Z} \rho(z) z$ with coefficients $\rho(z)\in\Bbbk$. Then $\rho$ satisfies \Cref{eq-rho}.
\end{example}

We recall from Section \ref{deligne-sect} that $\cR_t$ is the idempotent completion of $\cR_t^1$. In particular, recall the direct sum convention and the matrix notation for morphisms introduced in \Cref{matrix-notation}. 
In the following, we associate an idempotent $e_\rho$ in $\End([n]^{\oplus k})$ to any idempotent in the algebra $\Bbbk Z\otimes M_k(\Bbbk)$ (which uniquely corresponds to a map $\rho$ by \Cref{lem-rho-idempotent}). 

\begin{definition}[The idempotent $e_\rho=e_{\sigma,\rho}$] \label{e-def} We define a morphism $e_\rho=e_{\sigma,\rho}\in\End([n]^{\oplus k})$ by
\begin{equation} e_\rho=e_{\sigma,\rho} = \frac{1}{|Z|} \sum_{z\in Z} x_z \otimes \rho(z).
\end{equation}
\end{definition}

\begin{lemma}\label{e-idemp} The morphism $e_\rho$ is idempotent.
\end{lemma}

\begin{proof} We note that for any $g_1,g_2\in S_n$,
\[ x_{g_1} x_{g_2} = (x_{1_n} g_1) (x_{1_n} g_2) = x_{1_n} g_1 g_2 = x_{g_1 g_2}
.
\]
Hence, the assertion follows from \Cref{lem-rho-idempotent}.
\end{proof}

We can now determine the image of the objects $([n]^{\oplus k},e_\rho)$ just constructed under Deligne's functor $\cF_n$ to $\Rep(S_n)$ from Section \ref{functor-sect}.
For this, let us also consider the action of $\Bbbk Z\otimes M_k(\Bbbk)$ on $\Bbbk S_n\otimes \Bbbk^k$ or $\Bbbk Z\otimes \Bbbk^k$ defined by
\[ (z\otimes A)\cdot (g\otimes v) = gz^{-1} \otimes Av
\ .
\]
Let us temporarily denote the idempotent in $\Bbbk Z\otimes M_k(\Bbbk)$ associated with $\rho$ according to \Cref{lem-rho-idempotent} by $f$, that is,
\[ f = \frac{1}{|Z|} \sum_{z\in Z} z\otimes\rho(z)
\ .
\]
\begin{proposition}\label{inducedlift-prop} The image of $([n]^{\oplus k},e_\rho)$ under $\cF_n$ is isomorphic to the $S_n$-module
$
f \cdot(\Bbbk S_n\otimes\Bbbk^k)
$.
\end{proposition}
\begin{proof} Denote $e=e_{\rho}$. Recall from Section \ref{functor-sect}, that $\cF_n([n])\cong V_n^{\otimes n}$, where $V_n$ is the permutation module on $\Bbbk^{ n}$. In particular, we have seen in \Cref{rem-f-x-g-n} that $\cF_n(([n],x_{1_n}))\cong\Bbbk S_n$ and that $\cF(x_{1_g})$ is the left regular action for each $g\in S_n$. As $x_g=x_g x_{1_n}$ for all $g\in S_n$, we derive 
\[ \cF_n(([n]^{\otimes k},e)
 = \cF_n(e) (V_n^{\otimes n}\otimes \Bbbk^k)
 = \cF_n(e x_{1_n}) (V_n^{\otimes n}\otimes \Bbbk^k)
 \cong \cF_n(e) (\Bbbk S_n\otimes \Bbbk^k)
\]
and
\[ \cF_n(e)(g\otimes v) = \frac{1}{|Z|} \sum_{z\in Z} gz^{-1} \otimes \rho(z)v
\]
for all $g\in S_n$ and $v\in V$.
This proves the assertion.
\end{proof}

\begin{remark}\label{rem-image-of-e-rho} If $\rho$ is an algebra map, then the module $f\cdot(\Bbbk S_n\otimes \Bbbk^k)$ is just the induced module $\Bbbk S_n \otimes_{\Bbbk Z} \Bbbk^k$ and can be described more explicitly as follows: as a vector space, it is $\Bbbk S_n/Z \otimes \Bbbk^k$ and if $\{h_i\}_i$ is a fixed set of representatives of the cosets $S_n/Z$, then the $S_n$-action is given, for $g\in S_n$, by
\[ g (h_i\otimes v) = h_j \otimes \rho(z) v,
\]
where we can write $g h_i=h_j z$, with $z\in Z$, uniquely. This recovers \Cref{inducedaction}.
Hence, $\cF_n([n]^{\oplus k},e_\rho)\cong \Ind_Z^{S_n}(V)$, where $V$ is the $Z$-module defined by $\rho$.
\end{remark}

\subsection{Interpolation Objects in the Center}\label{interpolationsect}

We now extend the construction from Section \ref{liftinduced-sect} by constructing a lift of the half-braiding isomorphism $c_{V_n}$ from \Cref{ch}. This way, we obtain a large supply of objects in the monoidal center $\cZ(\cR_t)$.

To lift the braiding, we introduce convenient notation. Let $\pi_*^*=\pi_*\pi^*\in\End([1])$ be the morphism given by the partition $\{\{1\},\{1'\}\}$ and for any $1\leq i\leq n$, let $C_i\in\End([n+1])$ be the morphism given by the partition $\{\{1,1'\},\dots,\{i,i',n+1,(n+1)'\},\dots,\{n,n'\}\}$. We define
\begin{align}\label{Eij-def} E^i_j = C_j (1_n\otimes \pi^*_*) C_i=
\tikz[partition]{\scriptsize
\foreach \x in {1,...,4,4.8}  \makeDots{\x}{1};
\tpid{1,...,4}
\foreach \x in {1,2,3} \node at (\x+0.5,1.5){...};
\node[above] at (2,1){$i$};
\node[below] at (3,2){$j'$};
\node[right] at (4.8,1){$n+1$};
\node[right] at (4.8,2){$(n+1)'$};
\draw (4.8,1) to[bend left] (2,1);
\draw (4.8,2) to[bend right] (3,2);
}
\in \End([n+1])\, .
\end{align}
Hence, $E_j^i$ corresponds to the partition $\Set{\Set{1,1'},\ldots, \Set{i,i',n+1},\ldots, \Set{j,j',(n+1)'},\ldots,\Set{n,n'}}$ for all $i\neq j$ and to the partition $\Set{\Set{1,1'},\ldots,\Set{i,i',(n+1),(n+1)'},\ldots,\Set{n,n'}}$ if $i=j$.

\begin{definition}[The morphisms $d_1$, and $d_1^\rho=d_1^{\sigma,\rho}$] \label{def-d1-d1rho} Fix $n,\sigma,\rho$ as in the previous section. We define a morphism $d_1=d_1^\sigma\in\End([n+1]))$ by
\begin{equation}\label{d1dash} d_1 = 1_{n+1} + \sum_{i=1}^n (E^i_{\sigma(i)}-E^i_i)\, ,
\end{equation}
so we have $d_1^{\oplus k}\in\End([n]^{\oplus k}\otimes [1])\cong \End([n+1]^{\oplus k})\cong\End([n+1])\otimes M_k(\Bbbk)$ given by
\begin{equation} d_1^{\oplus k} = d_1\otimes I_k = \Big(1_{n+1} + \sum_{i=1}^n \big(E^i_{\sigma(i)}-E^i_i\big) \Big) \otimes I_k
\ ,
\end{equation}
where $I_k\in M_k(\Bbbk)$ is the identity matrix.

Further, we denote by $d_1^\rho=d_1^{\sigma,\rho}$ the restriction of $d_1$ to the object $([n]\otimes [1],e_{\mu}\otimes 1)$, using $e_\rho=e_{\sigma,\rho}$ from Definition \ref{e-def}. That is, $d_1^\rho=(e_{\rho}\otimes 1)d_1(e_{\rho}\otimes 1)$. 
\end{definition}

\begin{theorem}\label{halfbraiding-prop} The morphism $d_1^\rho$ induces a half-braiding $c^\rho$ for the object $([n]^{\oplus k},e_\rho)$ using Equations (\ref{inducedc})--(\ref{d1}), and hence defines an object $(([n]^{\oplus k},e_\rho),c^\rho)$ in $\cZ(\cR_t)$.
\end{theorem}

\begin{definition}[Interpolation Objects of the Center] \label{intobject-def}For $\sigma\in S_n$ and $\rho\colon Z\to M_k(\Bbbk)$ satisfying \Cref{eq-rho}, let $e_\rho$ be the idempotent from \Cref{e-def}. We denote the object of $\cZ(\cR_t)$ given by $([n]^{\oplus k}, e_\rho)$ together with the half-braiding $c^\rho$ obtained from $d_1^\rho$ (cf.~\Cref{halfbraiding-prop}), by $\uW{\sigma,\rho}$ (or $\uW{\rho}$ for short). We call objects of this form \emph{interpolation objects in $\cZ(\cR_t)$}.
\end{definition}


The following series of lemmas will result in the proof of \Cref{halfbraiding-prop}. We start by proving that $e_\rho$ annihilates morphisms coarser than $\pi_\sigma$.

\begin{lemma}\label{lem-annihilation}
Let $n\geq 2$, $\kappa\in S_n$ and $\lambda >\pi_\kappa$ be a partition. Then $e_{\rho}\lambda^{\oplus k}=\lambda^{\oplus k} e_{\rho}=0$.
\end{lemma}
\begin{proof}

Let $\tau\vdash n$ be a partition of $n$. Then we can define a partition $p_\tau$ in $P_{n,n}$ by partitioning $\Set{1,\ldots, n}$ according to $\tau$, and, in addition, requiring that $i$ and $i'$ are connected for all $i$.

For example, for $\tau=\Set{\Set{1,2},\Set{3,4,5}}$ we have
\begin{align*}
p_\tau=\tp{01,02,0,01,11,0,02,12,0,11,12, 0, 03,04,05,0,03,13,0,04,14,0,05,15,0,13,14,15}\, .
\end{align*}

Recalling that $x_\kappa=x_{1_n} \kappa=\kappa x_{1_n}$, this assertion is a consequence of the following claim.\newline
\emph{Claim.} Let $\tau>\omega$ be partitions of $n$. Then $x_{p_\omega} p_\tau=p_\tau x_{p_\omega}=0$.

We prove the claim by induction on the number of connected components $|\omega|$ of $\omega$. Note that $|\omega|\geq 2$ as $\tau$ is strictly coarser than $\omega$. If $|\omega|=2$, then $x_{p_\omega}=p_\omega-p_\tau$, since there is only one partition coarser than $\omega$ in this case. Now $p_\omega p_\tau=p_\tau p_\omega=p_\tau^2=p_\tau$, which implies the claim.

Now let $|\omega|\geq 3$ and assume the claim holds for all $\omega'$ with $|\omega'|<|\omega|$. By assumption, there exist $i,j$ which are connected in $\tau$ but not connected in $\omega$. It follows that
\begin{align*}
x_{p_\omega}=p_\omega-\sum_\mu x_\mu-\sum_{\nu}x_\nu,
\end{align*}
where $\mu>p_\omega$ such that $i,j$ are connected in $\mu$, and $\nu>p_\omega$ such that $i,j$ are not connected in $\nu$. It follows by induction hypothesis that $x_\nu p_\tau=0$ for all such $\nu$. This implies
\begin{align*}
x_{p_\omega}p_\tau=p_\omega p_\tau-\sum_\mu x_\mu p_\tau.
\end{align*}
However, $\mu_0:=p_\omega p_\tau$ appears among the $\mu$. All other $\mu$ are strictly coarser than $\mu_0$, and hence $x_{\mu_0}=\mu_0-\sum_{\mu\neq \mu_0}x_\mu$. Therefore,
\begin{align*}
x_{p_\omega}p_\tau=\mu_0-x_{\mu_0}p_\tau-\sum_{\mu>\mu_0} x_\mu p_\tau=\mu_0-\mu_0p_\tau+\sum_{\mu>\mu_0}x_\mu p_\tau-\sum_{\mu>\mu_0} x_\mu p_\tau=0,
\end{align*}
using that $p_\tau^2=p_\tau$. The claim follows after using a similar argument to show that $p_\tau x_{p_\omega}=0$.
\end{proof}

\begin{lemma}\label{decommute} The endomorphism $e_\rho\otimes 1$ and $d_1$ in $\End([n]^{\oplus k}\otimes [1])\cong \End([n+1]^{\oplus k})$ commute.
\end{lemma}

\begin{proof} Since identity morphisms commute with arbitrary morphisms, it is enough to show that $x_z\otimes 1$ commutes with $\sum_i E^i_{\sigma(i)}$ and with $\sum_i E^i_i$ for any $z\in Z$. Since $x_z=x_{1_n} z$ and $E^i_j$ commutes with $x_{1_n}\otimes 1$ for arbitrary $1\leq i,j\leq n$, it is enough to show that $z\otimes 1$ commutes with each of the described sums.

Clearly, for any $g\in S_n$, the morphism $\pi_g\otimes 1$ commutes with the morphism $1_n\otimes\pi^*_*$. Now for any $i$,
\[ \pi_g C_i \pi_g\inv = C_{g(i)}
 \ ,
\]
since $C_i$ connects the $(n+1)$-th strand with the $i$-th strand, which is the $g(i)$-th strand, when pre-composed with the permutation morphism $g\inv$. Hence,
\[
 z \sum_i E^i_i z\inv = \sum_i E^{z(i)}_{z(i)} = \sum_i E^i_i
 \quad\textnormal{ and }\quad
 z \sum_i E^i_{\sigma(i)} z\inv = \sum_i E^{z(i)}_{z(\sigma(i))} = \sum_i E^i_{\sigma(i)}
\ ,
\]
since $z$ commutes with $\sigma$. This completes the proof.
\end{proof}

\begin{lemma} For $\sigma_1,\sigma_2\in S_n$, $d_1^{\sigma_2} d_1^{\sigma_1} (x_{1_n}\otimes 1) = d_1^{\sigma_2\sigma_1} (x_{1_n}\otimes 1)$.
\end{lemma}

\begin{proof} Note that by \Cref{lem-annihilation} we see
\begin{align}
E^a_{b}E^i_jx_{1_n}=\delta_{j,a}E^i_{l}.
\end{align}
This follows as if $j$ is not equal to $a$, then two different strands are connected, so the partition is annihilated by $x_{1_n}$. It hence follows that
\begin{align*}
d_1^{\sigma_2} d_1^{\sigma_1} (x_{1_n}\otimes 1)
&= \Big(1_{n+1}+ \sum_i \big(
E^i_{\sigma_1(i)} + E^i_{\sigma_2(i)} - 2E^i_i 
\\&~\phantom{=}~+
E^{\sigma_1(i)}_{\sigma_2\sigma_1(i)} E^i_{\sigma_1(i)}
- E^{\sigma_1(i)}_{\sigma_1(i)} E^i_{\sigma_1(i)}
+ E^i_i E^i_i
- E^i_{\sigma_2(i)} E^i_i
\big)\Big) (x_{1_n}\otimes 1)
\\
 &= \Big(1_{n+1} + \sum_i \big( E^i_{\sigma_2\sigma_1(i)} - E^i_i\big) \Big)(x_{1_n}\otimes 1)
 = d_1^{\sigma_2\sigma_1} (x_{1_n}\otimes 1)\ .&&\qedhere
 \end{align*}
\end{proof}

We can use the above Lemma to describe the inverse of $d_1^\rho$ explicitly.

\begin{lemma}\label{d1inverse}
The morphism $d_1^{\rho}=d_1^{\sigma,\rho}$ is invertible and its inverse is given by $d_1^{\sigma\inv,\rho}$.
\end{lemma}
\begin{proof}
Write $e=e_\rho$. Recall from the proof of \Cref{e-idemp} that $e=e'x_{1_n}$ for some element $e'$. The assertion follows now from the previous lemma.
\end{proof}

By \Cref{d2}, we have the morphism $d_2^\rho\in  \End([n]^{\oplus k},e_\rho)\otimes [2])$ given by
\begin{align*}
d_2^\rho=(1_n^{\oplus k}\otimes\pi_X)(d_1^\rho\otimes 1)(1_n^{\oplus k}\otimes\pi_X)(d_1^\rho\otimes 1)
\end{align*}
To describe $d_2^\rho$ more explicitly, we define the partitions
\begin{align}\label{Eijab-def}
E^{i,a}_{j,b}=
\tikz[partition]{\scriptsize
\foreach \x in {1,...,7,9}  \makeDots{2*\x}{1};
\tpid{2,4,6,8,10,12}
\foreach \x in {1,...,5} \node at (2*\x+1,1.5){...};
\node[above] at (4,1){$i$};
\node[below] at (6,2){$j'$};
\node[above] at (8,1){$a$};
\node[below] at (10,2){$b'$};
\node[right] at (14,1){$n+1$};
\node[right] at (14,2){$(n+1)'$};
\node[right] at (18,1){$n+2$};
\node[right] at (18,2){$(n+2)'$};
\draw (14,1) to[bend left] (4,1);
\draw (14,2) to[bend right] (6,2);
\draw (18,1) to[bend left] (8,1);
\draw (18,2) to[bend right] (10,2);
}
\ ,
\end{align}
for indices $i,j,a,b\in\Set{1,\ldots,n}$ which are not necessarily distinct.

We also want to consider partitions where $n+1$ and $(n+1)'$ form a connected component, or where $n+2$ and $(n+2)'$ form a connected component, and we will use symbols of the form $E^{|,a}_{|,b}$ or $E^{i,|}_{j,|}$ for these partitions. That is, for instance,
\begin{align}\label{E--ab-def}
E^{|,a}_{|,b}=
\tikz[partition]{\scriptsize
\foreach \x in {1,...,5,7}  \makeDots{2*\x}{1};
\tpid{2,4,6,8}
\foreach \x in {1,...,3} \node at (2*\x+1,1.5){...};
\node[above] at (4,1){$a$};
\node[below] at (6,2){$b'$};
\node[right] at (10,1){$n+1$};
\node[right] at (10,2){$(n+1)'$};
\node[right] at (14,1){$n+2$};
\node[right] at (14,2){$(n+2)'$};
\draw (10,1) to (10,2);
\draw (14,1) to[bend left] (4,1);
\draw (14,2) to[bend right] (6,2);
}
\ .
\end{align}

\begin{lemma} \label{lem-helper-d2}
For all $1\leq i,j,k,l\leq n$,
\begin{gather}
E^{i,|}_{j,|} = E^i_j\otimes 1,
\label{Eeq1}
\\
E^{|,k}_{|,l}=(1_n\otimes\pi_X)(E^k_l\otimes 1)(1_n\otimes\pi_X),
\label{Eeq2}
\\
E^{i,k}_{j,l}=(1_n\otimes\pi_X)(E^k_l\otimes 1)(1_n\otimes\pi_X)(E^i_j\otimes 1).\label{Eeq3}
\end{gather}
\end{lemma}
\begin{proof} \Cref{Eeq1} follows from $\pi_X^2=1_2$. \Cref{Eeq2} follows from \Cref{Eeq1} by conjugating the first identity with $(1_n\otimes\pi_X)$. Finally, using the first two identities, we see that the left-hand side of \Cref{Eeq3} is just $E^{|,k}_{|,l} E^{i,|}_{j,|}$, which equals the right-hand side.
\end{proof}

This lemma is used to derive the following expression of $d_2^\rho$ in terms of the morphism
$d_2$ in $\End([n]\otimes [2])$ defined by
\begin{align}
d_2
 &= 1_{n+2}
 +\! \sum_{i=1}^n\! \big(E^{i,|}_{\sigma(i),|} + E^{|,i}_{|,\sigma(i)} - E^{i,|}_{i,|} - E^{|,i}_{|,i}\big)
 +\!\sum_{i,j=1}^{n}\! \big(E^{i,j}_{\sigma(i),\sigma(j)}+E^{i,j}_{i,j} - E^{i,j}_{\sigma(i),j} - \!E^{i,j}_{i,\sigma(j)}\big).
\end{align}

\begin{lemma}\label{restrictiond2}
The morphism $d_2^{\oplus k}$ commutes with $e_\rho\otimes 1_2$ and $d^\rho_2=d_2(e_\rho\otimes 1_2)$.
\end{lemma}

\begin{proof} Define
$$
\tilde{d}_2:=(1_n\otimes\pi_X)(d_1\otimes 1)(1_n\otimes\pi_X)(d_1\otimes 1).
$$
The morphism $\tilde{d}_2^{\oplus k}$ commutes with $e_\rho\otimes 1_2$, since $d_1$ commutes with $e_\rho\otimes 1$. The claim follows since by \Cref{lem-helper-d2}, $\tilde{d}_2=d_2$.
\end{proof}

\begin{lemma}\label{lem-annihilation-EE}The equality
\begin{equation}
\big(E^{i,j}_{a,b}\big)^{\oplus k} (e_\rho\otimes\pi_H)
= 0
= (1_n\otimes\pi_H) \big(E^{a,b}_{i,j}\big)^{\oplus k} (e_\rho\otimes 1_2)
\end{equation}
holds whenever $1\leq i,j,a,b\leq n$ and $i\neq j$.
\end{lemma}

\begin{proof} The restriction of $E^{i,j}_{a,b} (1_n\otimes\pi_H)$ (or $(1_n\otimes\pi_H) E^{a,b}_{i,j}$) to the set $\Set{1,\ldots, n,1',\ldots, n'}$ is a partition coarser than $1_n$ as $i\neq j$ are connected. Hence, pre-composition with $e_\rho\otimes 1_2$ (or post-composition with $1_2\otimes e_\rho$) gives zero by \Cref{lem-annihilation}.
\end{proof}

\begin{proof}[Proof of \Cref{halfbraiding-prop}]
We check that the conditions described in \Cref{centerconditions} hold for $d_1^\rho$ and $d_2^\rho$. It suffices to consider the case $k=1$ with trivial $\rho(z)=1$ for all $z\in Z$. Hence, we write $e=e_\rho$, and have $d_1^{\oplus k}=d_1$, $d_2^{\oplus k}=d_2$.
We note that $d_2$ is defined in terms of $d_1$ as required in \Cref{d2}. It remains to verify the four conditions in Equations (\ref{centercond1})--(\ref{centercond2}).

First of all, we check naturality with respect to $\pi_*$, $\pi^*$. For this, note that $E^i_{j} (1_n\otimes \pi_*) = E^k_j(1_n\otimes \pi_*)$ for all $i,j,k$. Therefore,
\[ d_1 (1_n \otimes \pi_*)
 = 1_n \otimes \pi_* + \sum_{i=1}^n \big(E^i_{\sigma(i)}-E^i_i\big) (1_n\otimes\pi^*_*)
 = 1_n \otimes \pi_*
\]
and similarly, $(1_n\otimes\pi^*) d_1 = 1_n \otimes \pi^*$. Hence, \Cref{centercond1} follows using \Cref{Eij-def}.

To check naturality with $d_2$ we compute
\begin{align*}
&(1_n\otimes \pi_X)d_2(1_n\otimes \pi_X)=(1_n\otimes\pi_X)
\Big(
  1_{n+2}
  + \sum_i \big(E^{i,|}_{\sigma(i),|} + E^{|,i}_{|,\sigma(i)} - E^{i,|}_{i,|} - E^{|,i}_{|,i}\big) \\
&\phantom{=}~  +\sum_{i,j} \big(E^{i,j}_{\sigma(i),\sigma(j)}+E^{i,j}_{i,j} - E^{i,j}_{\sigma(i),j} - E^{i,j}_{i,\sigma(j)}\big)
\Big)
(1_n\otimes \pi_X)
\\
&=
  (1_n\otimes\pi_X)\Big(\!1_{n+2}
  +\! \sum_i \big(E^{|,i}_{|,\sigma(i)} + E^{i,|}_{\sigma(i),|} \!- E^{|,i}_{|,i} - E^{i,|}_{i,|}\big)
 +\!\sum_{i,j} \big(E^{j,i}_{\sigma(j),\sigma(i)}\!+E^{j,i}_{j,i} - E^{j,i}_{\sigma(j),i} \!- E^{j,i}_{j,\sigma(i)}\big)\!\Big)
\\&=
  \Big(1_{n+2}
  + \sum_i \big(E^{i,|}_{\sigma(i),|} + E^{|,i}_{|,\sigma(i)} - E^{i,|}_{i,|} - E^{|,i}_{|,i}\big)
  +\sum_{i,j} \big(E^{i,j}_{\sigma(i),\sigma(j)}+E^{i,j}_{i,j} - E^{i,j}_{\sigma(i),j}\! - E^{i,j}_{i,\sigma(j)}\big)\!\Big)
  =d_2,
\end{align*}
using reordering of indices in the last step. Hence  $d_2 (e\otimes\pi_X) = (e\otimes\pi_X) d_2$ follows using  \Cref{decommute}.

Finally, to check naturality of $d_2^\rho$ with respect to $\pi_H$, we may verify that
$ [1_n\otimes\pi_H,d_2](e\otimes 1_2)=0,$
where $[\cdot,\cdot]$ denotes the commutator using the composition.

We observe that $1_n\otimes\pi_H$ commutes with partitions of the form $E^{|,i}_{|,i}, E^{i,|}_{i,|}$ or $E^{i,j}_{i,j}$ directly. Out of the remaining terms of the form $E^{i,j}_{a,b}$ in $d_2$, we can neglect those for which $i\neq j$, if $\pi_H$ is applied from the top, and those for which $a\neq b$, if $\pi_H$ is applied from the bottom, according to \Cref{lem-annihilation-EE}. Finally, for all $i,j$,
\begin{align*}(1_n\otimes\pi_H)E^{i,|}_{j,|}=E^{i,|}_{j,|}(1_n\otimes\pi_H)=E^{i,i}_{j,i}, && (1_n\otimes\pi_H)E^{|,i}_{|,j}=E^{|,i}_{|,j}(1_n\otimes\pi_H)=E^{i,i}_{j,i}.\end{align*}
Hence, we compute
\begin{align*}
&[1_n\otimes\pi_H,d_2](e \otimes 1_2)=  \bigg(\Big[ 1_n\otimes\pi_H,
 1_{n+2}+\sum_i \big(E^{i,|}_{\sigma(i),|} + E^{|,i}_{|,\sigma(i)} - E^{i,|}_{i,|} - E^{|,i}_{|,i}\big)\Big]\\&
  \phantom{=}~+ \Big[1_n\otimes\pi_H, \sum_{i,j} \big(E^{i,j}_{\sigma(i),\sigma(j)}+E^{i,j}_{i,j} - E^{i,j}_{\sigma(i),j} - E^{i,j}_{i,\sigma(j)}\big)
 \Big]\bigg)(e\otimes 1_2)
\\
 &= \bigg[1_n\otimes\pi_H,
 \sum_i \Big(E^{i,|}_{\sigma(i),|} + E^{|,i}_{|,\sigma(i)}
  +\sum_{i,j} \big(E^{i,j}_{\sigma(i),\sigma(j)}- E^{i,j}_{\sigma(i),j} - E^{i,j}_{i,\sigma(j)}\big)\Big)
 \bigg](e\otimes 1_2)
 \\
&=
\Big(
 \sum_i \big(E^{i,i}_{\sigma(i),i} + E^{i,i}_{i,\sigma(i)}
  +E^{i,i}_{\sigma(i),\sigma(i)}- E^{i,i}_{\sigma(i),j} - E^{i,i}_{i,\sigma(i)}\big)
\\
&\phantom{=}~-
 \sum_i \big(E^{i,\sigma(i)}_{\sigma(i),\sigma(i)} + E^{\sigma(i),i}_{\sigma(i),\sigma(i)}
  + E^{i,i}_{\sigma(i),\sigma(i)}- E^{i,\sigma(i)}_{\sigma(i),\sigma(i)} - E^{\sigma(j),j}_{\sigma(j),\sigma(j)}\big)
\Big)(e\otimes 1_2)
 = 0
\ .
\end{align*}
This completes the proof.
\end{proof}

\begin{example} We choose $n=2$, $\mu=(2)$, $\sigma=(12)\in S_2$, $Z=\Set{1,(12)}$, $k=1$, and the trivial $Z$-character $\rho\colon \Bbbk Z\to M_1(\Bbbk)$ with $\rho((12))=1$. Then $e = (1_2+\pi_X-2\pi_H)/2$ and
\[
d_1^\rho  =\left(
\tp{1,11,0,2,12,0,3,13,0}
+\tp{11,1,3,0,2,12,13}
+\tp{1,11,13,0,12,2,3}
-\tp{11,1,3,13,11,0,2,12}
-\tp{1,11,0,12,2,3,13,12}\right)(e\otimes 1)
\
\]
defines the object $\uW{(12),\triv}$ in $\cZ(\cR_t)$ for any $t$. We may, for example, also use
$\tilde{e}=(1_2-\pi_X)/2$ which corresponds to the choice the sign $Z$-character, $\rho(12)=-1$, and the same morphism $d_1$ determining the half-braiding. This gives the object $\uW{(12),\mathrm{sign}}$ in $\cD_t$.
\end{example}

\begin{example}
Choose $n=3$, $\mu=(3)$, $\sigma=(123)\in S_3$. Then $Z=\Set{1,\sigma,\sigma^2}$, and for $a=0,1,2$ we may consider the irreducible $Z$-character such that $\rho_a(123)=e^{2a\pi i/3}$. Then for $\uW{(123),\rho_a}$,
\begin{align*}
    e_1^{\rho_a}=&\frac{1}{3}\left(\tp{1,11,0,2,12,0,3,13,0}-\tp{1,11,1,3,13,11,0,2,12,0}-\tp{1,11,1,2,12,11,0,3,13,0}-\tp{1,11,0,2,12,2,3,13,12,0}+2\,\tp{1,11,12,2,1,0,2,3,13,12,0}\right.\\
    &+e^{\frac{2a\pi i}{3}}\left(
    \tp{1,12,0,2,13,0,3,11,0}-\tp{1,12,13,2,1,0,3,11,0}-\tp{1,11,12,3,1,0,2,13,0}-\tp{1,12,0,2,3,13,11,2,0}+2\,\tp{1,11,12,2,1,0,2,3,13,12,0}
    \right)\\
    &\left.+e^{-\frac{2a\pi i}{3}}\left(
    \tp{1,13,0,2,11,0,3,12,0}-\tp{1,2,13,11,1,0,3,12,0}-\tp{1,13,0,2,3,12,11,2,0}-\tp{1,3,13,12,1,0,2,11,0}+2\,\tp{1,11,12,2,1,0,2,3,13,12,0}\right)\right)\, ,\\
    d_1^{\rho_a}=&\left(\tp{1,11,0,2,12,0,3,13,0,4,14,0}
+\tp{11,1,4,0,2,12,14,0,3,13}
+\tp{1,11,0,12,2,4,0,3,13,14,0}
+\tp{1,11,14,0,2,12,0,13,3,4,0}\right.\\
&\left.-\tp{11,1,4,14,11,0,2,12,0,3,13,0}
-\tp{1,11,0,12,2,4,14,12, 0,3,13,0}
-\tp{1,11,0,12,2,0,3,13,14,4,3,0}\right)(e_1^{\rho_a}\otimes 1)\, .
\end{align*}
\end{example}

We consider a pair of conjugate permutations, $\sigma_2 = \tau \sigma_1 \tau\inv$ for $\sigma_1,\sigma_2,\tau\in S_n$. Then the centralizer groups $Z(\sigma_1)$ and $Z(\sigma_2)$ are isomorphic, and in fact, $\tau(\cdot)\tau\inv$ is a group isomorphism from $Z(\sigma_1)$ to $Z(\sigma_2)$. Assume $\rho_i\colon Z(\sigma_i)\to M_k(\Bbbk)$ for $i=1,2$ are two maps satisfying \Cref{eq-rho}. We say that $(\sigma_1,\rho_1)$ and $(\sigma_2,\rho_2)$ are \emph{conjugate} if, in addition to $\tau$, there is an invertible matrix $A\in M_n(\Bbbk)^\times$ such that
\[ \rho_2( \tau z \tau\inv) = A \rho_1(z) A\inv
\qquad\text{ for all } z\in Z(\sigma_1).
\]

\begin{proposition}\label{equivalence-conj}
If $(\sigma_1,\rho_1)$ and $(\sigma_2,\rho_2)$ are conjugate, then the interpolation objects $\uW{\sigma_1,\rho_1}$ and $\uW{\sigma_2,\rho_2}$ in $\cZ(\uRep(S_t))$ are isomorphic.
\end{proposition}

\begin{proof}

We consider the morphisms
\[
\phi := e_{\rho_2}(\tau\otimes_\Bbbk A)e_{\rho_1}
\in\Hom( ([n]^{\oplus k},e_{\rho_1}), ([n]^{\oplus k},e_{\rho_2}) ),
\]
where $\otimes_\Bbbk$ denotes the $\Bbbk$-linear tensor product, not the internal tensor product of our monoidal category.

The pairs $(\sigma_1,\rho_1)$ and $(s_2,\rho_2)$ being conjugate implies that $\phi$ is invertible with the inverse morphism $e_{\rho_1}(\tau\inv\otimes A\inv)e_{\rho_2}$.

Also, we can see that $\phi$ is compatible with the half-braidings defined by $d_1^{\sigma_1,\rho_1}$ and $d_1^{\sigma_2,\rho_2}$, respectively: we recall from the proof of \Cref{decommute} that $g E^i_j=E^{g(i)}_{g(j)} g$ for all $g\in S_n$ and $1\leq i,j\leq n$, so
\begin{align*}
(\phi\otimes 1) d^{\sigma_1,\rho_1}
 &= e_{\rho_2}
 \big((\tau\otimes_\Bbbk A) \otimes 1\big)
  \Big(1_{n+1} + \sum_i \big(E^i_{\sigma_1(i)} - E^i_i\big)\otimes_\Bbbk I_n\Big) e_{\rho_1}
 \\
 &= e_{\rho_2}
  \Big(1_{n+1} + \sum_i \big(E^{g(i)}_{g\sigma_2(i)} - E^{g(i)}_{g(i)}\big)\otimes_\Bbbk I_n\Big)
  \big((\tau\otimes_\Bbbk A) \otimes 1\big)
  e_{\rho_1}
  \\
 &= e_{\rho_2}
  \Big(1_{n+1} + \sum_i\big( E^i_{\sigma_2(i)} - E^i_i\big)\otimes_\Bbbk I_n\Big)
  \big((\tau\otimes_\Bbbk A) \otimes 1\big)
  e_{\rho_1}
\\
 &= d^{\sigma_2,\rho_2} (\phi\otimes 1)&&\qedhere
\end{align*}
\end{proof}

\begin{remark}\label{uWrho}
By \Cref{equivalence-conj}, the interpolation object $\uW{\sigma,\rho}$ only depends on the cycle type $\mu$ of $\sigma$, and we may use the notation $\uW{\mu,\rho}$. As a special case, if $\rho$ corresponds to a $Z(\sigma)$-representation, then isomorphic representations yield isomorphic interpolation objects, just as in the classical situation of $S_n$-crossed modules (see \Cref{Sncenter}).
\end{remark}

\subsection{The Ribbon Category \texorpdfstring{$\cD_t$}{Dt}}\label{Dtribbonsect}

We now consider the monoidal subcategory $\cD_t$ of $\cZ(\cR_t)$ generated under tensor products and direct sums by the interpolation objects studied in the previous section. We show that $\cD_t$ is a ribbon category and, as we will see in \Cref{functorcenter}, this category can be viewed as an interpolation category for $\cZ(\Rep(S_n))\simeq \Rep (\Drin S_n)$.

\begin{definition}
Let $\cD_t^0$ denote the full subcategory of $\cZ(\cR_t)$ generated under tensor products by all objects of the form $\uW{\sigma,\rho}$ as in \Cref{intobject-def}. Moreover, denote by  $\cD_t^1$ the closure of $\cD_t^0$ under finite direct sums (biproducts), and $\cD_t$ the idempotent completion of $\cD_t^1$.
\end{definition}

By construction, $\cD_t$ is a monoidal category closed under direct sums. We show that $\cD_t$ is also closed under duals.

\begin{remark} Note that, by construction, for any map $\rho\colon\Bbbk Z\to M_k(\Bbbk)$ as in \Cref{eq-rho}, the object $\uW{\sigma,\rho}$ is a subobject of $\uW{\sigma,\rho_k}$ for the map $\rho_k \colon\Bbbk Z\to M_k(\Bbbk)$ which is the $k\times k$-identity matrix at any element in $Z$. This follows from observing $\rho$ and $\rho_k$ only affect the idempotent of $[n]^{\oplus k}$, whose image is the underlying object of $\uW{\sigma,\rho}$ and $\uW{\sigma,\rho_k}$, respectively, but not the half-braiding. Hence, $\cD_t$, is, in fact, already generated by objects of the form $\uW{\sigma,\rho_k}$.
\end{remark}

\begin{lemma}\label{duallemma}
Given an interpolation object $\uW{\sigma,\rho}$ in $\cD_t^0$, 
the dual object in $\cZ(\cR_t)$ is given by $\uW{\sigma,\rho^*}$, where 
\begin{align} \rho^*(z)=\rho(z^{-1})^t, &&\forall z\in Z.
\end{align}
Here, $(\cdot)^t$ denotes the usual transpose of a matrix.
\end{lemma}
\begin{proof}
Note that for $e_\rho=\sum_{z\in Z}x_z\otimes \rho(z)$, the idempotent describing the dual object in $\cR_t$ is given by the dual morphism, which is
\begin{align}
e_\rho^*=\sum_{z\in Z}x_{z}^*\otimes \rho(z)^t=\sum_{z\in Z}x_{z}\otimes \rho(z^{-1})^t=e_{\rho^*}.
\end{align}
This follows as for $\pi\in S_n$, the dual morphism is $\pi^{-1}$ and hence $x_\pi^*=x_{\pi^{-1}}$ follows from $x_\pi=\pi x_{1_n}$, while for a matrix $A$, the dual is $A^t$, cf. \Cref{deligne-sect}. .

The braiding on the dual is induced from $d_1^{\oplus k}(e_\rho^*\otimes 1)$.
To verify this, recall that by \Cref{centerdual}, the dual $(V^*,c^*)$ of an object $(V,c)$ in the center is obtained using the inverse of $c$.
It suffices here to consider $d_1^*=\Psi_{[1],[n]} c^*_1$. By Lemma \ref{d1inverse}, the inverse of $d_1^\rho$ is given by $(d_1^{\rho})^{-1}=f_1^{\oplus k}(e_\rho\otimes 1)$, where
\begin{align*}
f_1=1_{n+1} + \sum_{i=1}^n \big(E^i_{\sigma^{-1}(i)} - E^i_i\big).
\end{align*}
As $e_\rho\otimes 1$ commutes with $d_1^{\oplus k}$ by Lemma \ref{decommute}, it commutes with $f_1^{\oplus k}$. Hence,
\begin{align*}
(d_1^\rho)^*
&=(\ev_{[n]} \otimes \Psi_{[1],[n]})\left(1_n\otimes \left(f_1^{\oplus k}(e_\rho\otimes 1)\Psi_{[1],[n]}\right) \otimes 1_n\right)(1_{n+1}\otimes \coev_{[n]})\\
&=(\ev_{[n]}\otimes \Psi_{[1],[n]})\left(1_n\otimes \left(f_1^{\oplus k}\Psi_{[1],[n]}\right)\otimes 1_n\right)(1_{n+1}\otimes \coev_{[n]})(e_\rho^*\otimes 1).
\end{align*}
We see that
\begin{align*} 
(\ev_{[n]}\otimes \Psi_{[1],[n]})\left(1_n\otimes \left(E^i_j\Psi_{[1],[n]}\right)\otimes 1_n\right)(1_{n+1}\otimes \coev_{[n]})= E^{j}_i,
\end{align*}
and hence
\begin{align*} (d_1^\rho)^* =
&\Big(1_{n+1}+\sum_{i=1}^n \big(E^{\sigma^{-1}(i)}_{i} - E^i_i\big)\Big)^{\oplus k} (e^*_\rho\otimes 1)=d_1^{\oplus k}(e^*_\rho\otimes 1).
\end{align*}
The claim follows.
\end{proof}

\begin{corollary}
The category $\cD_t$ is a pivotal category.
\end{corollary}
\begin{proof}
It follows from \cite{TV}*{Section 5.2.2} that $\cZ(\cR_t)$ is a pivotal category. Now, Lemma \ref{duallemma} implies that the subcategory $\cD_t$ is closed under left duals, which coincide with right duals, and hence $\cD_t$ is also pivotal.
\end{proof}

\begin{theorem}\label{ribbon-thm}
The category $\cD_t$ is a ribbon category.
Given an interpolation object $V=\uW{\sigma,\rho}$ in $\cZ(\cR_t)$, the twist is given by
\begin{align}
\theta_{V}=\big(\sigma^{-1}\big)^{\oplus k}e_\rho.
\end{align}
\end{theorem}
\begin{proof}
Recall from \Cref{centerbraiding} that in $\cZ(\cR_t)$, the braiding on an object $(V,c)$ is
$\Psi_{(V,c),(V,c)}=c_V.$ Hence, as $\cZ(\cR_t)$ pivotal, the left twist $\theta^l=\theta^l_{(V,c)}$ and right twist $\theta^r=\theta^r_{(V,c)}$ are given by
\begin{gather}\label{ltwistcenter}
\theta^l=(\ev^l_V\otimes \ide_V)(\ide_{V^*}\otimes c_{V})(\coev^r_V\otimes \ide_V), \\\theta^r=(\ide_V\otimes \ev^r_V)(c_V\otimes \ide_{V^*})(\ide_V\otimes \coev^l_V)\label{rtwistcenter}.
\end{gather}
Recall from \Cref{def-d1-d1rho} that
$$d_1^\rho=d_1^{\oplus k}(e_\rho\otimes 1)=(e_\rho\otimes 1)d_1^{\oplus k}$$
and recall that the half braiding $c^\rho_{[n]}$ is computed from $d_1^\rho$ via Equations (\ref{inducedc})--(\ref{d1}).
In particular, $c^\rho_{[n]}=c_{[n]}(e_\rho\otimes 1)$, where we can use
\begin{align*}
c_{[n]}=(1_{n-1}\otimes c_1)\ldots(1\otimes c_1\otimes 1_{n-2})(c_1\otimes 1_{n-1}).
\end{align*}
In the following computation, for the sake of easier notation, we use the morphism
$
d_{[n]}=\Psi_{[n],[n]}c_{[n]}.
$. Indeed, we claim that
\begin{align}\label{lefttwist-comp}
\theta^l_V=(\ev_{[n]}\otimes 1_n)(1_n\otimes c_{[n]})(1_n\otimes e_\rho\otimes e_\rho)(\coev_{[n]}\otimes 1_n)=\sigma^{-1}e_\rho.
\end{align}
To prove this claim, we introduce some notation. Let ${\mathbf{i}}=(i_1,\ldots, i_m)$ and ${\mathbf{j}}=(j_1,\ldots, j_m)$ be strings of indices of length $0\leq m\leq n$, where for any $p\in \Set{1,\ldots, m}$
\begin{itemize}
\item $i_p\in \Set{1,\ldots,n}$ or $j_p=|$;
\item if $i_p\in \Set{1,\ldots,n}$, then $j_p\in \Set{i_p,i_{\sigma(p)}}$; if $i_p=|$, then $j_p=|$.
\end{itemize}
We define the elements $E_{\mathbf{j}}^{\mathbf{i}}\in\End([n+m])$ recursively by $E^\varnothing_\varnothing=1_n$ and
\begin{align}
E_{\mathbf{j}}^{\mathbf{i}}
=(1_n\otimes \Psi_{[1],[m-1]}) \big(E_{j_n}^{i_n}\otimes 1_{n-1}\big) (1_n\otimes \Psi_{[m-1],[1]}) \big(E_{\hat{\mathbf{j}}}^{\hat{\mathbf{i}}}\otimes 1\big).
\end{align}
Here, ${\hat{\mathbf{i}}}=(i_1,\ldots, i_{m-1})$, ${\hat{\mathbf{j}}}=(j_1,\ldots, j_{m-1})$, and $E_{j_m}^{i_m}$ is defined as in \Cref{Eij-def} or if $i_m=j_m=|$, then $E_{j_m}^{i_m}=1_{n+1}$. Note that $E_{\mathbf{j}}^{\mathbf{i}}$ is an iteration of the notation in \Cref{Eijab-def} and \Cref{E--ab-def}.

Now the general form of summands appearing in the expansion of $\theta^l_V$ according to \Cref{lefttwist-comp} is
$$L_{\mathbf{j}}^{\mathbf{i}}:=(\ev_{[n]}\otimes 1_n)\big(1_n\otimes \Psi_{[n],[n]}E^{\mathbf{i}}_{\mathbf{j}}\big)(1_n\otimes e_\rho\otimes e_\rho)(\coev_{[n]}\otimes 1_n),$$
where $\mathbf{i},\mathbf{j}$ are of length $n$.

We first note that for $n=1$,
\begin{align*}
L_|^|
=(\ev_{[1]}\otimes 1)\big(1\otimes \Psi_{[1],[1]}E_|^|\big)(\coev_{[1]}\otimes 1) e_\rho
=(\ev_{[1]}\otimes 1)\big(1\otimes \Psi_{[1],[1]}E_1^1\big)(\coev_{[1]}\otimes 1) e_\rho
=L_{1}^1.
\end{align*}
Hence, inductively, we see that
\begin{align}\label{help-eqn}
L_{\mathbf{i}}^{\mathbf{i}}=L_{\mathbf{i}'}^{\mathbf{i}'}
\end{align}
where if $i_p=j_p=|$ for some $p$, then $i'_p=j'_p=p$. It therefore suffices to consider indices $\mathbf{i},\mathbf{j}\in \Set{1,\ldots, n}^n$.

Now consider the case where $\mathbf{i},\mathbf{j} \in \Set{1,\ldots, n}^n$. We observe that $L_{\mathbf{j}}^{\mathbf{i}}$ vanishes unless $\mathbf{j}=(1,\ldots, n)$. Indeed, if two indices repeat in either $\mathbf{i}$ or $\mathbf{j}$, then $L_{\mathbf{j}}^{\mathbf{i}}=0$ by \Cref{lem-annihilation}. Now assume that $j_p=a$, and $a\neq p$. Let us view $E^\mathbf{i}_\mathbf{j}$ as a subdiagram of a diagram corresponding to $L^\mathbf{i}_\mathbf{j}$ thought of as a composition as in \Cref{lefttwist-comp}, and let us temporarily use indices only referring to the vertices of this subdiagram. Then the vertices $a'$ and $(n+p)'$ are connected in the subdiagram. But $(n+p)'$ also connects to $p$, which, in turn, connects to $p'$ in the diagram of $L^\mathbf{i}_\mathbf{j}$. Thus, $a'$ and $p'$ are in the same connected component, and hence, again using \Cref{lem-annihilation}, $L_{\mathbf{j}}^{\mathbf{i}}=0$ as claimed. So we may assume $\mathbf{j}=(1,\ldots, n)$.

We recall that
$$d_1^{\oplus k}=\Big(1_{n+1} + \sum_{i=1}^n \big(E^i_{\sigma(i)}-E^i_i\big)\Big)\otimes I_k,$$
by \Cref{d1dash} and we note that if we consider any fixed orbit of the action of $\sigma$ on $\Set{1,\ldots,n}$, then $i_p=\sigma^{-1}(p)$ or $i_p=p$ simultaneously for all $p$ in this orbit.

We now claim that in the expansion of $\theta^l_V$ according to \Cref{lefttwist-comp},
$L_{\bf{j}}^{\bf i}$ has a zero coefficient if there exists an index $p\in\{1,\ldots, n\}$ such that $i_p=p\neq \sigma(p)$. Indeed, if such an index exists, then it appears, using \Cref{help-eqn}, with coefficient zero. This follows as both $E_{\mathbf{j}}^{{\mathbf{i}}}$ and $E_{\mathbf{j}'}^{{\mathbf{i}'}}$ contribute such a term, where $j_q=j_q'$, $i_q=i_q'$ for all $q\neq p$, but $j_p=p=i_p$ while $j'_p=|=i_p'$. The coefficient of
$E_{\mathbf{j}}^{{\mathbf{i}}}$ is the negative of the coefficient of $E_{\mathbf{j}'}^{{\mathbf{i}'}}$ according to the definition of $d_1$.

Hence, only $L_{(1,\ldots, n)}^{(\sigma^{-1}(1),\ldots, \sigma^{-1}(n))}$ has a non-zero coefficient in $\theta_V^l$. Its coefficient is equal to one.
We note that
$$L_{(1,\ldots, n)}^{(\sigma^{-1}(1),\ldots, \sigma^{-1}(n))}=\sigma^{-1}e_\rho$$
to complete the computation of $\theta^l_V$.

It is shown in a similar manner that $\theta^l_V=\sigma^{-1}e_\rho$. For this, we define
$$R_{\mathbf{j}}^{\mathbf{i}}:=(1_n\otimes \ev_{[n]})(\Psi_{[n],[n]}E^{\mathbf{i}}_{\mathbf{j}}\otimes 1_n)(1_n\otimes e_\rho\otimes e_\rho)(1_n\otimes\coev_{[n]}).$$
Similar combinatorial considerations as above show that $R_{\mathbf{j}}^{\mathbf{i}}=0$ unless
\begin{itemize}
\item $\mathbf{i}=(1,\ldots,n)$
\item $j_p$ is either always $p$ or $\sigma(p)$, for each orbit of $\sigma$ acting on $\Set{1,\ldots,n}$.
\item If $j_p=p\neq\sigma(p)$ for some $p$, then the coefficient of $R_{\mathbf{j}}^{\mathbf{i}}$ in $\theta^r_V$ is zero.
\end{itemize}
Thus, the computation of $\theta^r_V$ is completed by verifying that
$$L_{(\sigma(1),\ldots, \sigma(n))}^{(1,\ldots, n)}=\sigma^{-1}e_\rho.$$
Hence, we see that $\theta^l_V=\theta^r_V$, and we denote this twist by $\theta_V$.
\end{proof}

\begin{corollary}
An interpolation object $\uW{\sigma,\rho}$ has trivial twist iff \begin{align}
\rho(\sigma z)&=\rho(z), &\forall z\in Z.
\end{align}
\end{corollary}
\begin{proof}
Using the definition of $e_\rho$ in \Cref{e-def},
\[
(\sigma\inv)^{\oplus k} e_\rho = e_\rho
~~\Longleftrightarrow~~
\frac{1}{|Z|} \sum_z x_{\sigma\inv z}\otimes\rho(z)
= \frac{1}{|Z|} \sum_z x_{z}\otimes\rho(z)
~~\Longleftrightarrow~~
\rho(\sigma z) = \rho(z)
\]
for all $z\in Z$, as the elements $(x_g)_{g\in S_n}$ are linearly independent.
\end{proof}

In particular, choosing $\rho$ to correspond to the trivial $Z$-character, the resulting object has trivial twist. We can compute the dimensions of interpolation objects  more explicitly.

\begin{proposition}$~$
\begin{enumerate}
    \item[(i)] Let $g\in S_n$, then $\tr(x_g)=\delta_{g,1_n}\prod_{i=0}^{n-1}(t-i)$.
    \item[(ii)] Let $f\in \End([n])$ and $A\in M_k(\Bbbk)$, then
    $\tr(f\otimes A)=\tr(f)\tr A $, where $\tr A$ is the usual matrix trace.
    \item[(iii)] Let $\uW{\sigma,\rho}$ be an interpolation object as in \Cref{intobject-def}, then
    \begin{equation}
        \dim(\uW{\sigma,\rho})= \frac{1}{|Z|}\prod_{i=0}^{n-1}(t-i)\tr \rho(1).
    \end{equation}
\end{enumerate}
\end{proposition}
\begin{proof}\begin{enumerate}
    \item[(i)] We first observe that for $n\geq 0$,
    \begin{align}
        x_{1_{n+1}}=\Big(1_{n+1}-\sum_{i=1}^nE_i^i\Big)\left(x_{1_n}\otimes 1\right).
    \end{align}
    We now show that $\tr(x_{1_n})=\prod_{i=0}^{n-1}(t-i)$ by induction on $n$. The statement is clear for $n=0$, using the convention that the empty product equals $1$. For $n\geq 0$, it follows that
    \begin{align*}
        \tr(x_{1_{n+1}})&=\tr\left(x_{1_n}\otimes 1\right)-\tr\Big(\sum_{i=1}^nE_i^i \left(x_{1_n}\otimes 1\right)\Big)\\
        &=t\tr(x_{1_n})-\tr\Big(\sum_{i=1}^n\pi_{(i,n)} E_i^i \left(x_{1_n}\otimes 1\right)\pi_{(i,{n})}\Big)\\
        &=t\tr(x_{1_n})-\tr\Big(\sum_{i=1}^n\pi_{(i,n)} E_i^i\pi_{(i,n)} \left(x_{1_n}\otimes 1\right)\Big)\\
        &=t\tr(x_{1_n})-n\tr\Big(E_n^n \left(x_{1_n}\otimes 1\right)\Big)\\
        &=t\tr(x_{1_n})-n\tr\left(x_{1_n}\right)\\
        &=(t-n)\prod_{i=0}^{n-1}(t-i)=\prod_{i=0}^{n}(t-i).
    \end{align*}
    In the second equality, $\pi_{(i,n)}$ is the partition associated to the transposition $(i,n)$, and we use the cyclic invariance of the trace. The last equality uses the induction step.
    
    Now let $g\neq 1_n\in S_n$. Then $x_g=x_1g$. Let $i$ be such that $g(i)\neq i$. Then taking the trace connects $\Set{i,g(i)}$. Hence, by \Cref{lem-annihilation}, $\tr(x_g)=0$.
    \item[(ii)] Let $A$ be a $k\times k$-matrix, and $f\in \End([n])$. Then
    \begin{align*}
        \tr(f\otimes A)&= \ev_{[n]}^{\oplus k}\coev\left((f\otimes A)\otimes \ide\right)\coev_{[n]}^{\oplus k}\\
        &=\sum_{i=1}^k \ev_{[n]}\left((f\otimes A_{ii})\otimes \ide\right)\coev_{[n]}\\
        &=\tr(f)\tr(A).
    \end{align*}
    \item[(iii)] It follows directly from Part (i)--(ii) that
    \begin{align*}
        \dim(\uW{\sigma,\rho})=\tr e_\rho
        &=\frac{1}{|Z|}\sum_{z\in Z}\tr\left(x_z\otimes \rho(z)\right)
        =\frac{1}{|Z|}\prod_{i=0}^{n-1}(t-i)\tr \rho(1).&&\qedhere
    \end{align*}
\end{enumerate}
\end{proof}

\begin{remark}[The symmetric part of $\cD_t$]\label{easypart}
The category $\cR_t$ is symmetric monoidal. Hence, for any object $X\in \cR_t$, we obtain an object $(X,\Psi_{X,-})$ in $\cZ(\cR_t)$, with half-braiding  given by symmetric braiding $\Psi_{X,Y}$ of $\cR_t$, for $Y\in \cR_t$. All of these objects are contained in $\cD_t$. Indeed, let $n=1$ and $\sigma=1$ the unique element of $S_1$. In this case, $x_1=1$ and $d_1^\sigma=1_2$. The resulting interpolation object is just $([1],\Psi_{[1],-})$. The $n$-fold tensor product of this object displays $[n]$ with symmetric half-braiding as an object in $\cD_t^0$. As $\cD_t$ is the idempotent completion of $\cD_t^0$, we obtain that the image of $\cR_t$ inside of $\cZ(\cR_t)$ is contained in $\cD_t$. However, $\cD_t$ contains a wealth of other objects with non-symmetric half-braidings.
\end{remark}

To conclude this section, we mention the following open questions.

\begin{question}$~$\label{Dtquestions}
Is the monoidal center $\cZ(\cR_t)$ equivalent to $\cD_t$?
\end{question}

\begin{question}
What are the simple objects in $\cD_t?$ If $t$ is generic, is $\cZ(\cR_t)$ (or $\cD_t$) a semisimple category? Note that objects generating $\cD_t$ are parametrized by idempotent elements in the semisimple algebras $\Bbbk Z\otimes M_k(\Bbbk)$. By  \Cref{easypart} it follows that $\cD_t$ is \emph{not} semisimple of $t\in \mZ_{\geq 0}$.
\end{question}

Note that the indecomposable objects in $\uRep(S_t)$ were classified in \cite{CO}*{Theorem 3.7}. Up to isomorphism, the indecomposable objects are parametrized by the set of all partitions.

\subsection{A Functor to Crossed Modules}\label{functorcenter}

An explanation of the functor $\cF_n$ from Section \ref{functor-sect} is given by the process of \emph{semisimplification} of monoidal categories. This procedure is described in \cite{CO}*{Section 3.4}, it is a special case of a more general construction due to Barrett--Westbury (\cite{BW}, see also \cite{EGNO}*{Section 18.8}) valid for any spherical category $\cS$. We recall that a morphism $f\colon X\to Y$ in $\cS$ is \emph{negligible} if $\tr(fg)=0$ for all morphisms $g\colon Y\to X$. The set of negligible morphisms $\cN$ forms a tensor ideal in any spherical monoidal category.

We now extend Deligne's functor $\cF_n$ to the center. We use the following general fact. Recall that a \emph{tensor ideal} $\cI$ in a monoidal category $\cC$ is a collection of subsets $\cI(V,W)\subseteq \Hom(V,W)$ closed under 2-sided composition and tensor products by general morphisms in $\cC$. The \emph{quotient category} $\cC/\cI$ is a monoidal category having the same objects, but morphisms
$$\Hom_{\cC/\cI}(V,W)=\Hom(V,W)/\cI(V,W).$$ There is a natural quotient functor $\cQ\colon \cC\to \cC/\cI$

\begin{proposition}\label{centerlifts-prop}
Let $\cC$ be a monoidal category and $\cI$ a tensor ideal in $\cC$. Then the quotient functor induces a functor of braided monoidal categories
$$\cQ\colon \cZ(\cC)\longrightarrow\cZ(\cC/\cI).$$
\end{proposition}
\begin{proof}
The quotient category $\cC/\cI$ has the same objects as $\cC$. Let $(V,c)$ be an object in $\cZ(\cC)$. Given an object $M$ of $\cC$, consider $\cQ(c_M)\colon \cQ(V\otimes M)\to \cQ(M\otimes V)$. Fixing a choice of structural natural isomorphism $\mu^{\cQ}_{V,W}\colon \cQ(V\otimes W)\isomorph \cQ(V)\otimes \cQ(W)$, we produce a natural isomorphism $c_M'\colon \cQ(V)\otimes \cQ(M)\isomorph \cQ(M)\otimes \cQ(V)$ which still satisfies the axioms of a half-braiding. Here, we identify objects of $\cC$ and $\cC/\cI$.
\end{proof}

Let us now turn to the monoidal category $\cR_n$ and the tensor ideal $\cN$ formed by its negligible morphisms.
\begin{theorem}[\cite{Del}, Th\'eor\`eme 6.2]
For $n\in \mZ_{\geq 0}$, there is an equivalence of categories $\cR_n/\cN\cong \Rep(S_n)$. Under this equivalence, the functor $\cF_n$ is natural isomorphic to the quotient functor $\cQ$.
\end{theorem}

\begin{proposition} Let $\mu\vdash n$ and assume that $\rho\colon Z\to M_k(\Bbbk)$ corresponds to a simple $Z(\mu)$-module $V$ as in \Cref{ex-rho-irrep}. Then the induced functor $\cF_n\colon \cZ(\uRep(S_n))\to\cZ(\Rep(S_n))$ sends the interpolation object  $\uW{\mu,\rho}$ from \Cref{intobject-def} and \Cref{uWrho} to the simple object $W_{\mu,\rho}$ of \Cref{Sncenter-thm}.
\end{proposition}

\begin{proof} Let $\sigma\in S_n$ be of cycle type $\mu$ and let us denote $V=\Bbbk^k$. If $\rho$ is a irreducible module over $Z=Z(\mu)$ as assumed, then, by \Cref{inducedlift-prop} and \Cref{rem-image-of-e-rho}, 
$\cF_n(([n]^{\oplus k},e))\cong W_{\mu,\rho}=\Bbbk S_n\otimes_{\Bbbk Z} V=\Ind_Z^{S_n}(V)$ is the induced module, as desired.

It remains to check that $\cF_n(d_1^{\rho})$ recovers the morphism $c_{V_n}$ from \Cref{ch}. Indeed, it follows from \Cref{piimage} and \Cref{ximage} that $d_1$ corresponds to the morphism $V_n^{\otimes (n+1)}\to V_n^{\otimes (n+1)}$ given by
\begin{align*}
v_{i_1,\ldots, i_n}\otimes v_j\longmapsto \begin{cases}v_{i_1,\ldots, i_n}\otimes  v_{g\sigma g\inv(j)},& \text{if } (i_1,\ldots, i_n)\text{ are pairwise distinct}, \\
0, & \text{otherwise}.
\end{cases}
\end{align*}
where $v_{i_1,\ldots, i_n}=v_{i_1}\otimes \ldots \otimes v_{i_n}\in V_n^{\otimes n}$, $g=\begin{pmatrix} 1 &\dots& n \\ i_1 & \dots & i_n \end{pmatrix}\in S_n$,  and $v_j\in V_n$. 

Now $\cF_n(x_{1_n})(v_{i_1,\dots,i_n})$ corresponds to $g$ when identifying $\cF_n(x_{1_n})(V_n)$ with $\Bbbk S_n$. Hence, for any $v\in V$, $\cF_n(e)(v_{i_1,\dots,i_n}\otimes v)$ corresponds to $g\otimes v$ when identifying $\cF_n(([n]^{\oplus k}, e_\rho))$ with $\Bbbk S_n\otimes_{\Bbbk Z} V$ (see \Cref{inducedlift-prop} and \Cref{rem-image-of-e-rho}). In particular, this is an element of $S_n$-degree $g\sigma g\inv$ (see \Cref{Sncenter}).

Hence, under the above isomorphism, $$\Psi_{W,V_n}\circ \cF_n(d_1^\rho)((g\otimes v)\otimes v_j)=c_{V_n}((g\otimes v)\otimes v_j),$$ where $\Psi$ denotes the symmetry of $\Rep(S_n)$.
\end{proof}

\begin{theorem}\label{maintheorem1}
The restriction of the braided pivotal functor $\cQ\colon \cZ(\uRep(S_n))\to \cZ(\Rep(S_n))$ to $\cD_n$ is essentially surjective and full on morphism spaces.
\end{theorem}
\begin{proof}
We have seen in \Cref{centerlifts-prop} that any simple object in $\cZ(\Rep(S_n)$ is isomorphic to an object in the image of an object of $\cD_n$ under $\cF_n$. As $\cZ(\Rep(S_n))$ is semisimple, this also implies fullness by Schur's Lemma.
\end{proof}

\begin{remark} By virtue of being a braided pivotal functor, $\cQ$ maps the twist of $\cD_n$ to the usual twist of $\cZ(\Rep(S_n))$.
\end{remark}

The following lemma and corollary were added to answer a question raised by an anonymous referee. Here, we use the \emph{semisimplification} of a Karoubian $\Bbbk$-linear monoidal pivotal category of \cite{EO}*{Section 2.3} (which generalizes \cite{BW}), i.e., the quotient by the tensor ideal of negligible morphisms.

\begin{lemma} If $\cC$ is a $\Bbbk$-linear Karoubian pivotal braided monoidal category whose unit object has endomorphism ring isomorphic to $\Bbbk$, then the negligible morphisms in $\cZ(\cC)$ form the unique maximal tensor ideal $\cN$ of $\cZ(\cC)$.
If $\cQ \colon\cZ(\cC)\to\cD$ is an essentially surjective full braided pivotal functor, where $\cD$ is a semisimple pivotal braided monoidal category, then $\cD$ is equivalent to the semisimplification of $\cZ(\cC)$ and
$\cQ$ corresponds to the quotient functor by $\cN$ under this equivalence.  
\end{lemma}

\begin{proof} The category $\cZ(\cC)$ is a rigid braided category, whose tensor unit is given by the tensor unit in $\cC$ with half-braiding corresponding to identities, so the endomorphism ring of the unit in $\cZ(\cC)$ is $\Bbbk$. Hence, by \cite{Cou}*{Lemma 2.5.3} (see also \cite{AK}*{Proposition 7.1.4}), the ideal formed by all negligible morphisms is the unique maximal tensor ideal in $\cZ(\cC)$.

The kernel of the monoidal functor $\cQ$ is a tensor ideal in $\cZ(\cC)$ giving a semisimple quotient category. Hence, this kernel contains and thus equals the ideal of negligible morphisms by maximality. Note that by construction, $\cZ(\cC)$ is a $\Bbbk$-linear Karoubian pivotal monoidal category whose nilpotent endomorphisms have trace zero, since their images under $\cQ$ have trace zero in the semisimple category $\cD$, and $\cQ$ preserves the trace. Therefore, $\cD$ is equivalent to the semisimplification of $\cZ(\cC)$ in the sense of \cite{EO}*{Section 2.3}. 
\end{proof}

\begin{corollary} \label{cor-semisimplification}
The category $\cZ(\Rep(S_n))$ is the semisimplification of both $\cZ(\uRep(S_n))$ and $\cD_n$, and the tensor ideal of negligible morphisms is the unique maximal tensor ideal in $\cZ(\uRep(S_n))$.
\end{corollary}

\begin{proof} This follows directly from the previous lemma and a parallel argument for $\cD_n$.
\end{proof}


\section{Polynomial Interpolation Invariants of  Framed Ribbon Links}\label{invariantsect}

As the main application of this work, we now use $\cD_t$ to associate polynomials in a variable $\mathbf{t}$ to framed ribbon links. These polynomials are invariants of such framed ribbon links.

\subsection{Ribbon Invariants from Ribbon Categories}

It is well-known (see e.g. \cite{TV}*{Section 3.3}) that a ribbon category $\cR$ gives invariants of $\cR$-coloured ribbon diagrams under three Reidemeister moves (where the first Reidermeister move is a weaker version of the classical move from link theory), see \Cref{fig-reide}.

Following   \cite{EGNO}*{Section 8.10} (see also \cite{RT2}*{Section 4}) we consider the category of \emph{framed ribbon tangles}. Let $S_{m,n}$ be a disjoint union of $m$ ribbons $S_1\times [0,1]$ and $n$ squares $[0,1]\times [0,1]$. A smooth non-intersecting embedding of $S_{m,n}$ in $\mR^2\times[0,1]$ is called a \emph{ribbon tangle} if, for any of the squares, it sends the points $[0,1]\times \Set{0}$ to $\mR^2\times \Set{0}$, and $[0,1]\times \Set{1}$ to $\mR^2\times \Set{1}$, and the images of the tubes do not meet $\mR^2\times \Set{0,1}$. A ribbon tangle is \emph{framed} if $S_{n,m}$ has an orientation, hence giving an orientation to its image in $\mR^2\times [0,1]$. We call the line segments of the images of squares which are contained in $\mR^2\times \Set{0}$ the \emph{inputs} and those in $\mR^2\times \Set{1}$ the \emph{outputs} of the tangle. There is a natural structure of a ribbon category, denoted by $\cal{FRT}$, on ribbon tangles. The objects are integers $k,l\geq 0$, and
$$
\Hom_{\cal{FRT}}(k,l)=\bigslant{\Set{\text{framed ribbon tangles with $k$ inputs and $l$ outputs}}}{\text{isotopy}}.
$$
There is a natural composition operation connecting the output and inputs of tangles. The tensor product on $\cal{FRT}$ is given by the addition of integers.  The twist in $\cal{FRT}$ is the twist of an identity strand on the object $1$ as in \Cref{fig-twist}.

\begin{center}
\begin{figure}[hbt]
\begin{align*}
    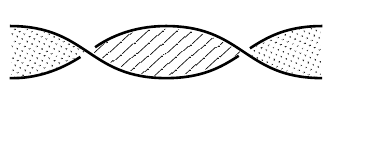
\end{align*}
\caption{The twist in $\cal{FRT}$}
    \label{fig-twist}
\end{figure}
\end{center}

\begin{theorem}[See \cites{FY,RT2,JS1,Sh}]\label{ribboninvariants}
Given any ribbon category $\cR$ and an object $X$ of $\cR$, there exists a canonical monoidal functor
\begin{align*}
\mathrm{Inv}_{X}\colon \cal{FRT}\longrightarrow \cR,
\end{align*}
which sends the object $1$ to the object $X$.
\end{theorem}

In other words, $\cal{FRT}$ is \emph{free} as a ribbon category. Applying $\mathrm{Inv}_X$ to a framed ribbon tangle $\cL$ with no inputs and outputs, $\mathrm{Inv}_X(\cL)$ gives an element in the base field $\Bbbk$. This element is an \emph{invariant} of $\cL$ under the three Reidemeister moves for ribbons displayed in \Cref{fig-reide}. We call such a framed ribbon tangle with no inputs and outputs a \emph{(framed) ribbon link}, and if it has only one connected component, a \emph{(framed) ribbon knot}.

Given a ribbon link, we obtain a \emph{link} by restricting to $S_1\times \Set{1/2}$. This gives a mapping from the set of ribbon links to that of links which respects equivalence. If the ribbon category $\cR$ has trivial twist on $X$, i.e. $\theta_X=1$, then $\mathrm{Inv}_X$ is an invariant of links.

We note that in \cites{RT2,RT1} (see also \cite{TV}*{Section 15.1.3}) a more general category of \emph{coloured} framed ribbon tangles (called HCDR-graphs) is defined, where, in addition, labels by objects and morphisms from a ribbon category appear. Ribbon categories also associate invariants to such coloured tangles.

\begin{figure}[hbt]
\begin{gather*}
~\vcenter{\hbox{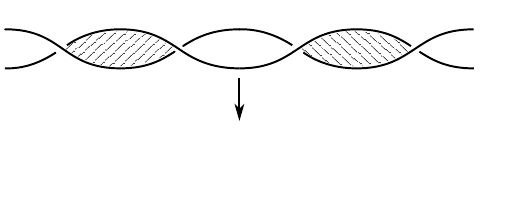}}
~ \vcenter{\hbox{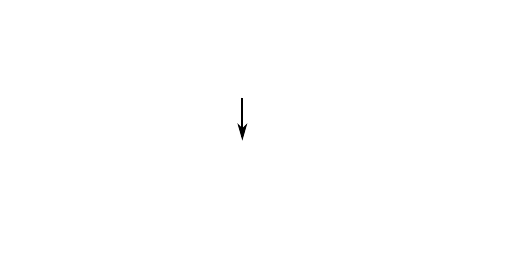}}~
\vcenter{\hbox{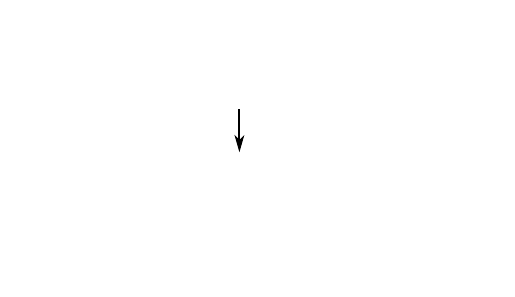}}
\end{gather*}
\caption{Three ribbon Reidemeister moves}
    \label{fig-reide}
\end{figure}

\subsection{Dijkgraaf--Witten Invariants}

Let $\cL\subset \mR^3$ be a \emph{tame} link. That is, the smooth embedding defining $\cL$ can be thickened to a tubular neighborhood, giving an embedding $S_1\times D_2\to \mR^3$, for a small ball $D_2$, with image $\ov{\cL}\subset \mR^3$.  Then the complement $\mR^3\setminus \ov{\cL}$ is a compact $3$-manifold. In fact, any compact $3$-manifold arises via such a \emph{surgery presentation}. The \emph{link group} of $\cL$ is $G_{\cL}:=\pi_1(\mR^3\setminus \ov{\cL})$.

In general, Dijkgraaf--Witten theory is based on a finite group $G$ together with a $3$-cocycle $\omega$. From this datum, one defines a quasi-Hopf algebra $\Drin^\omega G$ which is a twist of the Drinfeld double of $G$ \cite{DPR}. Dijkgraaf--Witten theory is a fully extended TQFT. Hence, in particular, it provides invariants for closed $3$-manifolds. These are related to invariants of links, see \cite{AC}.

 We shall restrict to the \emph{untwisted} Dijkgraaf--Witten theory, where $\omega$ is the trivial $3$-cocycle and $\Drin^\omega G =\Drin G $. In this case, the $3$-manifold invariant $Z^{\mathrm{DW}}_G(C)$ associated to $C=\mR^3\setminus \ov{\cL}$ is given by
\begin{align}
Z^{\mathrm{DW}}_G(C)=\frac{1}{|G|}\op{Inv}^{\mathrm{DW}}_G(\cL),
\end{align}
where $\op{Inv}^{\mathrm{DW}}_G(\cL)=|\Hom_{\mathrm{group}}(G_\cL,G)|$ is the number of group homomorphisms from the link group to $G$. As the notation suggests, $\op{Inv}^{\mathrm{DW}}_G(\cL)$ is an invariant of links.

The TQFT $Z_G^{\mathrm{DW}}$ can be computed from the representation theory of $\Drin G$ \cites{AC,FQ,Fre2}. Let $D=\Drin G^{\reg}$ denote the regular $\Drin G $-module. 
Then \Cref{ribboninvariants} provides the link invariants as $\op{Inv}_D(\cL)=\op{Inv}^{\mathrm{DW}}_G(\cL)$, see \cite{TV}*{Appendix~H} and \cites{Fre2, Che}.
 
The main application of the present work is to replace the module category of the group $G=S_n$ by Deligne's category $\uRep(S_t)$ and compute, by means of the ribbon categories $\cD_t$, invariants of ribbon links. This way, we obtain new ribbon link polynomials.

\subsection{A Lift of the Regular Module}

To provide lifts of untwisted Dijkgraaf--Witten invariants, we will now lift the regular module of $\Drin G$ to $\cD_t$. 
In \Cref{regulardec-lemma} we recalled that the regular $\Drin S_n$-module decomposes as a direct sum of modules $V^\mu$, indexed by partitions $\mu\vdash n$.
For given $n$ and a partition $\mu\vdash n$, we will construct an object $\uV{\mu}$ in $\cZ(\uRep(S_t))$, for any $t$, such that for $t=n$, we have $\cF_n(\uV{\mu})\cong V^{\mu}$. Hence, the sum of all $\uV{\mu}$ together gives a lift of the regular $\Drin S_n$-module to an object in $\cZ(\uRep(S_t))$.

\begin{definition}\label{uVmu-def}
Let $\sigma\in S_n$ be an element of cycle type $\mu$, $Z=Z(\sigma)$ its centralizer. We define
\begin{equation}
e_\mu=\frac{1}{|Z|}\sum_{z\in Z}x_z\otimes x_1.
\end{equation}
Then $e_\mu$ is an idempotent in $\End([n]\otimes [n])$ which is independent of the choice of $\sigma$.
Further, consider $d_1$ from \Cref{d1dash}. We denote
\begin{align}
d_1^\mu=(1_n\otimes \Psi_{[1],[n]})(d_1\otimes 1_{n})(1_n\otimes \Psi_{[n],[1]})(e_\mu\otimes 1),
\end{align}
and define $\uV{\mu}=([n]\otimes [n],e_\mu)$. Further, define $\uD{n}:=\oplus_{\mu\vdash n}\uV{\mu}$.
\end{definition}

\begin{lemma}\label{Vmu-lift}
The object $\uV{\mu}$, together with the half-braiding obtained from $d_1^{\mu}$ using Equations (\ref{inducedc})--(\ref{d1}) defines an object in $\cZ(\cR_t)$. There is an isomorphism $\cF_n(\uV{\mu})\cong V^\mu$ in $\cZ(\Rep(S_n))$. Hence, $\cF_n(\uD{n})$ is isomorphic to the regular $\Drin S_n$-module by \Cref{regulardec-lemma}.
\end{lemma}
\begin{proof}
The object $\uV{\mu}$ is the tensor product of $\uW{\sigma,\triv}$, where $\triv$ indicates the idempotent obtained from the trivial $Z$-module, and $([n],x_{1_n})$ (with trivial braiding, cf. \Cref{easypart}). This object is contained in $\cD_t$ since $[n]$ with trivial braiding is an object in $\cD_t^0$, and $x_{1_n}$ is an idempotent in $\cD_t^0$.
\end{proof}

We observe that by the computation in \Cref{ribbon-thm}, the objects $\uV{\mu}$ and $\uD{n}$ have trivial twists since
$$\theta_{\uV{\mu}}=(\sigma^{-1} \otimes 1_n)e_\mu=\frac{1}{|Z|}\sum_{z\in Z}x_z\sigma^{-1}\otimes x_1=e_\mu,$$
using reordering, $\sigma^{-1}\in Z$, and $x_z\sigma^{-1}=x_{z\sigma^{-1}}$. Hence, $\uV{\mu}$ and $\uD{n}$ provide link invariants.

\subsection{Ribbon Link Polynomials}

In this section, we define the (framed) ribbon link polynomials associated to objects of $\cD_t$.

We start by observing that Deligne's category $\uRep(S_{\mathbf{t}})$ --- together with the interpolation objects and the ribbon category $\cD_{\mathbf{t}}$ --- can be defined for a freely adjoint variable $\mathbf{t}$. 
In this version, we obtain a $\Bbbk[\mathbf{t}]$-linear ribbon category $\cD_{\mathbf{t}}$, and the category $\cD_t$ arise by specifying $\mathbf{t}\mapsto t\in \Bbbk$. The resulting link invariants obtained via \Cref{ribboninvariants} from $\cD_{\mathbf{t}}$ are now polynomials in $\mathbf{t}$.
It is a consequence of \Cref{ribbon-thm} that the category $\cD_{\mathbf{t}}$ provides invariants of ribbon links $\cL\subset \mR^3$.

\begin{definition}Let $\sigma\in S_n$ have cycle type $\mu$, 
let $\uW{\sigma,\rho}$ denote an interpolation object in $\cD_{\mathbf{t}}$ as in \Cref{intobject-def} and $\uD{n}$ an object as in \Cref{uVmu-def}, and let $\cL$ be a ribbon link. We denote
\begin{gather}
    \rP_{\mu,\rho}(\cL,{\mathbf{t}}):=\op{Inv}_{\uW{\sigma,\rho}}(\cL),\\
        \rP_{n}(\cL,{\mathbf{t}}):=\op{Inv}_{\uD{n}}(\cL).
\end{gather}
We refer to $\rP_{\mu,\rho}(\cL,\mathbf{t})$, $\rP_{n}(\cL,\mathbf{t})$ as \emph{ribbon link polynomials} associated to $(\mu, \rho)$, respectively, $S_n$.
\end{definition}

By \Cref{equivalence-conj}, $\rP_{\mu,\rho}$ is independent of the choice of $\sigma$ and $\rho$ up to conjugation.
It follows directly that evaluating the link polynomials at a specific value ${\mathbf{t}}\mapsto t\in \Bbbk$ gives the link invariants $\op{Inv}(\cL)$ for the corresponding object viewed in $\cD_t$. From Theorem \ref{ribbon-thm} we conclude:

\begin{corollary}
The polynomials $\rP_{\mu,\rho}(\cL,{\mathbf{t}})$ are invariants of ribbon links and the polynomials  $\rP_{n}(\cL,{\mathbf{t}})$ are invariants of links.
\end{corollary}

\subsection{Relation to Untwisted Dijkgraaf--Witten Invariants}\label{untwistedrel}

We start with a general observation about monoidal functors preserving invariants of ribbon links.
Then, using the functor $\cF_n$ from \Cref{functorcenter}, we explain the relationship between the ribbon link polynomials from the previous section and untwisted Dijkgraaf--Witten invariants.

\begin{lemma}
Let $\cR$, $\cS$ be $\Bbbk$-linear ribbon categories such that $\End_{\cS}(I)\cong \Bbbk$, and $\cF\colon \cR\to \cS$  be a braided monoidal functor preserving duals. Then for any ribbon link $\cL$,
\begin{align}
\cF(\mathrm{Inv}_X(\cL))=\mathrm{Inv}_{\cF(X)}(\cL).
\end{align}
\end{lemma}

Using this observation, together with \Cref{Vmu-lift}, we obtain the following result.

\begin{corollary}\label{DWinterpolation}
Let $\cL$ be a ribbon link. Then
evaluating the ribbon link polynomials at ${\mathbf{t}}=n$ recovers the untwisted Dijkgraaf--Witten invariants associated to $S_n$. In particular,
\begin{equation}
\rP_{\mu,\rho}(\cL,n)=\op{Inv}_{W_{\mu,\rho}}(\cL),\qquad\text{and}\qquad \rP_n(\cL,n)=\op{Inv}^{\mathrm{DW}}_{S_n}(\cL).
\end{equation}
\end{corollary}

\subsection{Examples of Link Polynomials}\label{invariantexamples}

In this section, we give examples of ribbon link polynomials. The computations were carried out using \emph{Wolfram Mathematica}\textsuperscript{\textregistered}. Given an element $\sigma\in S_n$ of cycle type $\mu$, $\rho\colon Z\to M_k(\Bbbk)$ satisfying \Cref{eq-rho}, and a (framed) ribbon Link $\cL$, recall that the associated ribbon link polynomial is denoted by $\rP_{\mu,\rho}(\cL,\mathbf{t})$.

We will restrict to the computation of invariants for \emph{ribbon torus links}. For this, denote by $S$ the generating object of the category $\mathcal{FRT}$. Given a pair of integers $(p,q)$, with $p\geq 1$, the $(p,q)$-ribbon torus link is defined as
\begin{align*}
  \cT_{p,q}= \ev^r_{S^{\otimes p}} (\Psi_{S,S^{\otimes (p-1)}}^{-q}\otimes \ide_{S^*}^{\otimes p})\coev^l_{S^{\otimes p}}.
\end{align*}
That is, $\cT_{p,q}$ is the braid closure of $p$ bands, which are braided $-q$ times using the crossing where the left-most band crosses above the other bands if $q\leq 0$ (or $q$ times crossing below the other bands if $q\geq 0$).
In this section all knot pictures are understood to have thick untwisted bands as strings.

\subsubsection*{\texorpdfstring{$2$}{2}-Cycle Invariants}

We now let $\mu=(2)$, $\sigma =(1,2)\in S_2$, and consider the associated invariants for small torus knots. In \Cref{S2table} all polynomials are divided by the dimension of the object, which is $\dim \uW{(2),\rho}=\tfrac{1}{2}\mathbf{t}(\mathbf{t}-1)$, for $\rho=\triv$ or $\mathrm{sign}$, where $\triv(12)=1$, $\mathrm{sign}(12)=-1$.

\begin{table}
{\renewcommand{\arraystretch}{1.5}
\begin{tabular}[htb]{c|c|c}
Ribbon torus link $\cT$&$\frac{\rP_{(2),\triv}(\cT,\mathbf{t})}{\dim \uW{(2),\triv}}$&$\frac{\rP_{(2),\mathrm{sign}}(\cT,\mathbf{t})}{\dim \uW{(2),\mathrm{sign}}}$\\[5pt] \hline
$\cT_{2,-2}=\torusknotinv{2}{2}$& $\frac{\mathbf{t}^2}{2}-\frac{5 \mathbf{t}}{2}+4$&$\frac{\mathbf{t}^2}{2}-\frac{5 \mathbf{t}}{2}+4$\\
$\cT_{2,-3}=\torusknotinv{2}{3}$&$2\mathbf{t}-3$& $-2\mathbf{t}+3$\\
$\cT_{2,-6}=\torusknotinv{2}{6}$&$\frac{\mathbf{t}^2}{2}-\frac{\mathbf{t}}{2}$&$\frac{\mathbf{t}^2}{2}-\frac{\mathbf{t}}{2}$\\
$\cT_{3,-4}=\torusknotinv{3}{4}$&$2 \mathbf{t}^2-8 \mathbf{t}+9$&$2 \mathbf{t}^2-8 \mathbf{t}+9$\\
$\cT_{4,-5}=\torusknotinv{4}{5}$&$2 \mathbf{t}^3-18 \mathbf{t}^2+52 \mathbf{t}-47$&$-2 \mathbf{t}^3+18 \mathbf{t}^2-52 \mathbf{t}+47$
\end{tabular}}\vspace{5pt}
\caption{Examples of $2$-cycle invariants}
\label{S2table}
\end{table}

Examples of ribbon link invariants for cycle type $(2)$ are summarized in Table \ref{S2table}. The invariants $\rP_{(2),\triv}(\cT,\mathbf{t})$ associated to the trivial character $\triv$ are link invariants as the twist is the identity on this object. However, the invariants $\rP_{(2),\mathrm{sign}}(\cT,\mathbf{t})$ which use the sign character $\mathrm{sign}$ are ribbon link invariants.

In addition, we note that the invariants $\rP_{(2),\triv}(\cT,\mathbf{t})$ for $\cT=\cT_{2,-5}$ and $\cT_{2,-7}$ are trivial (i.e. equal to the categorical dimension). Further, the knots $\cT_{2,-9}$ and $\cT_{3,-10}$ have the same invariant as the trefoil $\cT_{2,-3}$.

\subsubsection*{\texorpdfstring{$3$}{3}-Cycle Invariants}
The next case is where $\mu=(3)$ and $\sigma=(1,2,3)\in S_3$. In this case, there are three irreducible $Z(\mu)=C_3$-modules given by the third roots of unity. We indicate the irreducible characters by $\chi^{a}$, where
\begin{align}
    \chi^{a}(123)=e^{ \tfrac{2a\pi i}{3}}, &&a=0,1,2.
\end{align}
The corresponding ribbon link polynomials are denoted by $\rP_{(3),a}(\cT,\mathbf{t})$. \Cref{S3table} contains examples of these invariants. Note that $\rP_{(3),2}(\cT,\mathbf{t})$ is obtained from $\rP_{(3),1}(\cT,\mathbf{t})$ by replacing $e^{2i \pi /3}$ with $e^{-2i \pi/3}$.

\begin{table}
{\renewcommand{\arraystretch}{1.5}
\begin{tabular}[H]{c|c|c}
Ribbon torus link $\cT$&$\frac{\rP_{(3),0}(\cT,\mathbf{t})}{\dim \uW{(3),0}}$&$\frac{\rP_{(3),1}(\cT,\mathbf{t})}{\dim \uW{(3),1}}$\\[5pt] \hline
$\cT_{2,-2}=\torusknotinv{2}{2}$&$\frac{\mathbf{t}^3}{3}-4 \mathbf{t}^2+\frac{47 \mathbf{t}}{3}-18$&$\frac{\mathbf{t}^3}{3}-4 \mathbf{t}^2+\frac{47 \mathbf{t}}{3}+e^{\frac{2 i \pi }{3}}+e^{-\frac{2 i \pi }{3}}-20$\\
$\cT_{2,-3}=\torusknotinv{2}{3}$&$3 \mathbf{t}-8$&$3 \mathbf{t}-8$ \\
$\cT_{2,-5}=\torusknotinv{2}{5}$&$3 \mathbf{t}^2-21 \mathbf{t}+37$&$(3  \mathbf{t}^2-21 \mathbf{t}+36 +e^{\frac{2 i \pi }{3}})e^{\frac{2 i \pi }{3}}$ \\
$\cT_{2,-6}=\torusknotinv{2}{6}$&$\frac{\mathbf{t}^3}{3}-4 \mathbf{t}^2+\frac{56 \mathbf{t}}{3}-27$&$\frac{\mathbf{t}^3}{3}-4 \mathbf{t}^2+\frac{56 \mathbf{t}}{3}-27$\\
$\cT_{3,-4}=\torusknotinv{3}{4}$& $3\mathbf{t}^3-36\mathbf{t}^2+144\mathbf{t}-188$ &$\left(3\mathbf{t}^3-36\mathbf{t}^2+144\mathbf{t}-188\right)e^{-\frac{2i\pi}{3}}$\\
\end{tabular}}\vspace{5pt}
\caption{Examples of $3$-cycle invariants}
\label{S3table}
\end{table}

We note that the invariant $\rP_{(3),\triv}(\cT_{2,-5},\mathbf{t})$ is no longer trivial, but  $\rP_{(3),\triv}(\cT_{2,-7},\mathbf{t})$ is still trivial as in the $2$-cycle case. Also, $\cT_{2,-9}$ still has the same invariant as the trefoil knot.

\subsubsection*{Trefoil Invariants} \Cref{Trefoil} contains the knot invariants associated to the left-handed trefoil $\cT_{2,-3}$. Here, we consider the invariants corresponding to the object $\uW{\mu,\triv}$, where $\triv$ is the idempotent associated to the trivial $Z(\mu)$-module.

\begin{table}
{\renewcommand{\arraystretch}{1.5}
\begin{tabular}[H]{c|c|c|c|c|c}
Cycle type $\mu$ &$(1)$&$(2)$&$(3)$&$(4)$&$(2,2)$\\ \hline
$\frac{\rP_{\mu,\triv}(\cT_{2,-3},\mathbf{t})}{\dim \uW{\mu,\triv}}$ &$1$&$2 \mathbf{t}-3$&$3 \mathbf{t}-8$&$2  \mathbf{t}^2-16  \mathbf{t}+37$&$4 \mathbf{t}^2-28 \mathbf{t}+49$
\end{tabular}}\vspace{5pt}
\caption{Invariants of the left-handed trefoil knot}
\label{Trefoil}
\end{table}

We observe that all invariants computed so far are divisible by the dimension of the object.


\appendix

\section{Idempotent Completion and the Monoidal Center}\label{idempotent-appendix}

In the following, we consider a pre-additive monoidal category $\cC$.
We denote by $\cC'$ the \emph{idempotent completion} (or \emph{Karoubi envelope}) of $\cC$. This category can be described as having objects $X_e:=(X,e)$, where $X$ is an object of $\cC$ and $e\colon X\to X$ is an idempotent morphism, i.e. $e^2=e$. For two objects $X_e$ and $Y_f$, the morphism space is the $\Bbbk$-vector space
\begin{equation}
\Hom_{\cC'}(X_e, Y_f) =f\Hom_{\cC}(X ,Y)e.
\end{equation}
Now if $g\in\Hom_{\cC'}(X_e,X_e)$ is idempotent, then $g=ege$ and this morphism splits as
\begin{equation}
\xymatrix{
X_e\ar[dr]_{g=ge}\ar[rr]^{g=ege}&&X_e\\
&X_g\ar[ur]_{g=eg}&
}
\end{equation}
The $\Bbbk$-linear category $\cC'$ is again monoidal with the tensor product defined by
\begin{equation}
X_e\otimes Y_f= (X\otimes Y)_{e\otimes f}.
\end{equation}
If $\cC$ is braided with braiding $\Psi$, then so is $\cC'$ with the braiding $\Psi_{X_e,Y_f}$ given by
\begin{equation}
(f\otimes e)\Psi_{X,Y}(e\otimes f)=(f\otimes e)\Psi_{X,Y},
\end{equation}
by naturality of $\Psi$. We identify $\cC$ with the monoidal full subcategory on objects of the form $X_{\ide_X}$. 

\begin{proposition}\label{idempotentcenter-prop}
The idempotent completion $\cZ(\cC)'$ of $\cZ(\cC)$ is isomorphic to a full braided monoidal subcategory of the center $\cZ(\cC')$ of the idempotent completion $\cC'$ of $\cC$.
\end{proposition}
\begin{proof}
First, we show that $\cZ(\cC)$ is a full subcategory of $\cZ(\cC')$. An object $(V,c)$ in $\cZ(\cC)$ gives an object $(V,c')$ in $\cZ(\cC')$, where the half-braiding is defined as
\begin{align}
c'_{X_e}=(e\otimes \ide_V)c_X(\ide_V\otimes e)=(e\otimes \ide_V)c_X=c_X(\ide_V\otimes e),
\end{align}
using naturality of $c$ and that $e$ is an idempotent. Further, for a morphism $g\colon X_e\to Y_f$,
\begin{align*}
c'_{Y_f}(\ide_V\otimes g)&=(f \otimes \ide_V)c_Y(\ide_V\otimes fge)\\&=(fge\otimes \ide_V)c_X=(g\otimes \ide_V)c'_{X_e},
\end{align*}
which shows naturality of $c'$. Note that the tensor compatibility of the half-braiding as in \Cref{tensorcomp} is clear by the definition of the tensor product in $\cC'$:
\begin{align*}
c'_{X_e\otimes Y_f}
&=(e\otimes f \otimes \ide_V)c_{X\otimes Y}( \ide_V \otimes e\otimes f)\\
&=(e\otimes f \otimes \ide_V)(\ide_V\otimes c_Y)(c_{X}\otimes \ide_V)( \ide_V \otimes e\otimes f)\\
&=
(e\otimes (f\otimes \ide_V)c_{Y})((e\otimes \ide_V)c_X\otimes f)\\
&=(\ide_{X_e}\otimes c'_{Y_f})(c'_{X_e}\otimes \ide_{Y_f}).
\end{align*}

Clearly, any morphism $\lambda\colon (V,c)\to (W,d)$ in $\cZ(\cC)$ induces a morphism $(V,c')\to (W,d')$ in $\cZ(\cC')$. Further, the functor is full. Indeed, any morphism $\lambda\colon (V,c')\to (W,d')$ in $\cZ(\cC')$ is given by a morphism $\lambda\colon V\to W$ in $\cC$ which commutes with the natural isomorphisms $c'_{X_e}$, $d'_{X_e}$  for any object $X_e$ of $\cC'$. In particular, this naturality holds with respect to any object $X_{\ide_X}$, and hence $\lambda$ is induced from a morphism in $\cZ(\cC)$.

Next, we prove that $\cZ(\cC')$ is idempotent complete. Let $g\colon (V_e,c)\to(V_e,c)$ be an idempotent in $\cZ(\cC')$. Then $g$ can be viewed as an idempotent endomorphism $g$ of $V$ in $\cC$ satisfying $g=ege$. For any object $X_h$, consider the restriction $d_{X_h}\colon  V_g\otimes X_h\to X_h\otimes V_g$ of $c_{X_h}$ given by
$$d_{X_h}=(\ide_{X_h}\otimes g)c_{X_h}(g\otimes \ide_{X_h})=(
\ide_{X_h}\otimes g)c_{X_h}.$$ Then $d$ is clearly still a natural isomorphism. The morphism $g=ge\colon V_e\to V_g$ becomes a morphism $(V_e,c)\to (V_g, d)$ using that $g$ commutes with the half-braiding $c$.
Similarly, $g=eg\colon V_g\to V_e$ is a morphism from $(V_g,d)$ to $(V_e, c)$ in the center of $\cC$. Hence, the idempotent $g$ splits in $\cZ(\cC)$. Hence, by the universal property of the idempotent completion of $\cZ(\cC)$, we obtain a canonical faithful additive functor $\iota\colon \cZ(\cC)'\to \cZ(\cC')$. 

The functor $\iota$ is also full. Consider a morphism $g\colon (V_e,c')\to (W_f,d')$ in $\cZ(\cC')$, where $(V_e,c')$ and $(W_f,d')$ are in the image of $\iota$, i.e.~$e\colon (V,c)\to (V,c)$ and $f\colon (W,d)\to (W,d)$ are idempotent morphisms in $\cZ(\cC)$, and $c'$, $d'$ are obtained by restricting the half-braidings $c$ and $d$ along $e$ and $f$, respectively. Then $g\colon V\to W$ is a morphism in $\cC$ such that $g=fge$, and for any object $X$ of $\cC$,
\begin{align*}
d'_X(g\otimes \ide_X)=(\ide_X\otimes g)c'_X,
\end{align*}
which implies
\begin{align*}
d_X(g\otimes \ide_X)=d_X(f\otimes \ide_X)(g\otimes \ide_X)=(\ide_X\otimes g)(\ide_X\otimes e)c_X=(\ide_X\otimes g)c_X.
\end{align*}
Thus, $g\colon (V,c)\to (W,d)$ is a morphism in $\cZ(\cC)'$, and $\iota$ is full.
It follows from construction that the functor $\iota$ is monoidal and preserves the braiding. 
\end{proof}

Note that if $\cC$ is idempotent complete, then there is an canonical equivalence $\cZ(\cC)'\simeq \cZ(\cC')$.


\bibliography{biblio}
\bibliographystyle{amsrefs}

\end{document}

%% file: Graphics/twist.pdf_tex
\begingroup%
  \makeatletter%
  \providecommand\color[2][]{%
    \errmessage{(Inkscape) Color is used for the text in Inkscape, but the package 'color.sty' is not loaded}%
    \renewcommand\color[2][]{}%
  }%
  \providecommand\transparent[1]{%
    \errmessage{(Inkscape) Transparency is used (non-zero) for the text in Inkscape, but the package 'transparent.sty' is not loaded}%
    \renewcommand\transparent[1]{}%
  }%
  \providecommand\rotatebox[2]{#2}%
  \newcommand*\fsize{\dimexpr\f@size pt\relax}%
  \newcommand*\lineheight[1]{\fontsize{\fsize}{#1\fsize}\selectfont}%
  \ifx\svgwidth\undefined%
    \setlength{\unitlength}{107.43652857bp}%
    \ifx\svgscale\undefined%
      \relax%
    \else%
      \setlength{\unitlength}{\unitlength * \real{\svgscale}}%
    \fi%
  \else%
    \setlength{\unitlength}{\svgwidth}%
  \fi%
  \global\let\svgwidth\undefined%
  \global\let\svgscale\undefined%
  \makeatother%
  \begin{picture}(1,0.39610816)%
    \lineheight{1}%
    \setlength\tabcolsep{0pt}%
    \put(0,0){\includegraphics[width=\unitlength,page=1]{twist.pdf}}%
    \put(-0.00772622,0.01368006){\color[rgb]{0,0,0}\makebox(0,0)[lt]{\lineheight{1.25}\smash{\begin{tabular}[t]{l}$0$\end{tabular}}}}%
    \put(0,0){\includegraphics[width=\unitlength,page=2]{twist.pdf}}%
    \put(0.82997774,0.01367995){\color[rgb]{0,0,0}\makebox(0,0)[lt]{\lineheight{1.25}\smash{\begin{tabular}[t]{l}$1$\end{tabular}}}}%
    \put(0,0){\includegraphics[width=\unitlength,page=3]{twist.pdf}}%
  \end{picture}%
\endgroup%

%% file: Graphics/Reidemeister1s.pdf_tex
\begingroup%
  \makeatletter%
  \providecommand\color[2][]{%
    \errmessage{(Inkscape) Color is used for the text in Inkscape, but the package 'color.sty' is not loaded}%
    \renewcommand\color[2][]{}%
  }%
  \providecommand\transparent[1]{%
    \errmessage{(Inkscape) Transparency is used (non-zero) for the text in Inkscape, but the package 'transparent.sty' is not loaded}%
    \renewcommand\transparent[1]{}%
  }%
  \providecommand\rotatebox[2]{#2}%
  \newcommand*\fsize{\dimexpr\f@size pt\relax}%
  \newcommand*\lineheight[1]{\fontsize{\fsize}{#1\fsize}\selectfont}%
  \ifx\svgwidth\undefined%
    \setlength{\unitlength}{148.07739902bp}%
    \ifx\svgscale\undefined%
      \relax%
    \else%
      \setlength{\unitlength}{\unitlength * \real{\svgscale}}%
    \fi%
  \else%
    \setlength{\unitlength}{\svgwidth}%
  \fi%
  \global\let\svgwidth\undefined%
  \global\let\svgscale\undefined%
  \makeatother%
  \begin{picture}(1,0.42546464)%
    \lineheight{1}%
    \setlength\tabcolsep{0pt}%
    \put(0,0){\includegraphics[width=\unitlength,page=1]{Reidemeister1s.pdf}}%
    \put(-0.00420428,0.00744405){\color[rgb]{0,0,0}\makebox(0,0)[lt]{\lineheight{1.25}\smash{\begin{tabular}[t]{l}$0$\end{tabular}}}}%
    \put(0,0){\includegraphics[width=\unitlength,page=2]{Reidemeister1s.pdf}}%
    \put(0.90748115,0.00744405){\color[rgb]{0,0,0}\makebox(0,0)[lt]{\lineheight{1.25}\smash{\begin{tabular}[t]{l}$1$\end{tabular}}}}%
    \put(0,0){\includegraphics[width=\unitlength,page=3]{Reidemeister1s.pdf}}%
  \end{picture}%
\endgroup%

%% file: Graphics/Reidemeister2s.pdf_tex
\begingroup%
  \makeatletter%
  \providecommand\color[2][]{%
    \errmessage{(Inkscape) Color is used for the text in Inkscape, but the package 'color.sty' is not loaded}%
    \renewcommand\color[2][]{}%
  }%
  \providecommand\transparent[1]{%
    \errmessage{(Inkscape) Transparency is used (non-zero) for the text in Inkscape, but the package 'transparent.sty' is not loaded}%
    \renewcommand\transparent[1]{}%
  }%
  \providecommand\rotatebox[2]{#2}%
  \newcommand*\fsize{\dimexpr\f@size pt\relax}%
  \newcommand*\lineheight[1]{\fontsize{\fsize}{#1\fsize}\selectfont}%
  \ifx\svgwidth\undefined%
    \setlength{\unitlength}{148.36150539bp}%
    \ifx\svgscale\undefined%
      \relax%
    \else%
      \setlength{\unitlength}{\unitlength * \real{\svgscale}}%
    \fi%
  \else%
    \setlength{\unitlength}{\svgwidth}%
  \fi%
  \global\let\svgwidth\undefined%
  \global\let\svgscale\undefined%
  \makeatother%
  \begin{picture}(1,0.50047818)%
    \lineheight{1}%
    \setlength\tabcolsep{0pt}%
    \put(0,0){\includegraphics[width=\unitlength,page=1]{Reidemeister2s.pdf}}%
    \put(-0.00419623,0.00742979){\color[rgb]{0,0,0}\makebox(0,0)[lt]{\lineheight{1.25}\smash{\begin{tabular}[t]{l}$0$\end{tabular}}}}%
    \put(0,0){\includegraphics[width=\unitlength,page=2]{Reidemeister2s.pdf}}%
    \put(0.90765832,0.00742979){\color[rgb]{0,0,0}\makebox(0,0)[lt]{\lineheight{1.25}\smash{\begin{tabular}[t]{l}$1$\end{tabular}}}}%
    \put(0,0){\includegraphics[width=\unitlength,page=3]{Reidemeister2s.pdf}}%
  \end{picture}%
\endgroup%

%% file: Graphics/Reidemeister3s.pdf_tex
\begingroup%
  \makeatletter%
  \providecommand\color[2][]{%
    \errmessage{(Inkscape) Color is used for the text in Inkscape, but the package 'color.sty' is not loaded}%
    \renewcommand\color[2][]{}%
  }%
  \providecommand\transparent[1]{%
    \errmessage{(Inkscape) Transparency is used (non-zero) for the text in Inkscape, but the package 'transparent.sty' is not loaded}%
    \renewcommand\transparent[1]{}%
  }%
  \providecommand\rotatebox[2]{#2}%
  \newcommand*\fsize{\dimexpr\f@size pt\relax}%
  \newcommand*\lineheight[1]{\fontsize{\fsize}{#1\fsize}\selectfont}%
  \ifx\svgwidth\undefined%
    \setlength{\unitlength}{148.07739902bp}%
    \ifx\svgscale\undefined%
      \relax%
    \else%
      \setlength{\unitlength}{\unitlength * \real{\svgscale}}%
    \fi%
  \else%
    \setlength{\unitlength}{\svgwidth}%
  \fi%
  \global\let\svgwidth\undefined%
  \global\let\svgscale\undefined%
  \makeatother%
  \begin{picture}(1,0.5774122)%
    \lineheight{1}%
    \setlength\tabcolsep{0pt}%
    \put(0,0){\includegraphics[width=\unitlength,page=1]{Reidemeister3s.pdf}}%
    \put(-0.00420428,0.00744405){\color[rgb]{0,0,0}\makebox(0,0)[lt]{\lineheight{1.25}\smash{\begin{tabular}[t]{l}$0$\end{tabular}}}}%
    \put(0,0){\includegraphics[width=\unitlength,page=2]{Reidemeister3s.pdf}}%
    \put(0.90748115,0.00744405){\color[rgb]{0,0,0}\makebox(0,0)[lt]{\lineheight{1.25}\smash{\begin{tabular}[t]{l}$1$\end{tabular}}}}%
    \put(0,0){\includegraphics[width=\unitlength,page=3]{Reidemeister3s.pdf}}%
  \end{picture}%
\endgroup%